\documentclass[11pt]{amsart}
\input epsf
\usepackage{latexsym}
\usepackage{amssymb,amsmath,amsthm,amscd, amssymb,
graphicx, enumerate, verbatim}
\usepackage[all]{xy}

\numberwithin{equation}{section}
\newtheorem{cor}[equation]{Corollary}
\newtheorem{lem}[equation]{Lemma}
\newtheorem{prop}[equation]{Proposition}
\newtheorem{thm}[equation]{Theorem}
\newtheorem{quest}[equation]{Question}

\newtheorem{Example}[equation]{Example}
\newenvironment{ex}{\begin{Example}\rm}{\end{Example}}
\newtheorem{remark}[equation]{Remark}
\newenvironment{rmk}{\begin{remark}\rm}{\end{remark}}
\def\co{\colon\thinspace}

\newcommand{\Iso}{\mbox{Iso}}

\newcommand{\e}{\varepsilon}

\def\g{\gamma}

\def\bh{\mathbf h}
\def\bv{\mathbf v}
\def\bg{\mathbf g}
\def\d{\partial}

\def\r{\rho}

\def\s{\sigma}
\def\l{\lambda}

\def\Z{\mathbb{Z}}
\def\S1{\bf S^1}

\newcommand{\ddr}{{\d_r}}

\newcommand{\chn}{{\mathbf{CH}^n}}

\newcommand{\chm}{{\mathbf{CH}^{n-1}}}

\newcommand{\cV}{{\mathcal V}}
\newcommand{\cH}{{\mathcal H}}

\begin{document}

\abovedisplayskip=6pt plus3pt minus3pt
\belowdisplayskip=6pt plus3pt minus3pt

\title[Complex hyperbolic hyperplane complements]
{\bf Complex hyperbolic hyperplane complements}
\rm


\author{Igor Belegradek}
\address{Igor Belegradek\\School of Mathematics\\ Georgia Institute of
Technology\\ Atlanta, GA 30332-0160}\email{ib@math.gatech.edu}

\date{}
\begin{abstract} 
We study spaces obtained from a complete finite volume 
complex hyperbolic $n$-manifold $M$ by removing a compact 
totally geodesic complex $(n-1)$-submanifold $S$.
The main result is that  
the fundamental group of $M\setminus S$ is relatively hyperbolic, 
relative to fundamental groups of the ends of $M\setminus S$,
and $M\setminus S$ admits a complete finite volume
$A$-regular Riemannian metric of negative sectional curvature.

It follows that for $n>1$ the fundamental group of $M\setminus S$
satisfies Mostow-type Rigidity, has solvable word and conjugacy problems,
has finite asymptotic dimension and rapid decay property,
satisfies Borel and Baum-Connes conjectures,
is co-Hopf and residually hyperbolic,
has no nontrivial subgroups with property (T), and 
has finite outer automorphism group. Furthermore, 
if $M$ is compact, then the fundamental group of $M\setminus S$
is biautomatic and satisfies Strong Tits Alternative.
\end{abstract}
\maketitle
%
\section{Introduction}

Let $M$ be a (connected) complete finite volume
complex hyperbolic $n$-manifold, 
and let $S$ be a (possibly disconnected) 
compact totally geodesic complex submanifold
of dimension $(n-1)$; so
the pair $(M,S)$ is modelled on $(\chn, \chm)$
where $\chn$ denotes the complex hyperbolic symmetric space
of dimension $n$.

This paper is a systematic study of $M\setminus S$,
the manifold obtained by ``drilling'' $S$ in $M$.
Clearly $M\setminus S$ can be identified 
with the interior of a compact manifold $N$ that
is obtained from $M$ by removing a tubular neighborhood of $S$
and chopping off all cusps (in case $M$ is noncompact). 
There are two kinds of components of $\d N$:
compact infranil manifolds appearing as cusp cross-sections of $M$,
and circle bundles over components of $S$. 

It is known that
$N$ is an aspherical manifold with 
$\pi_1$-incompressible boundary~\cite[Lemma B.1]{Bel-rh-warp}.
As noted after Corollary~\ref{thm: intro-gr-theoretic-cor}, the group $\pi_1(N)$ 
shares various rigidity properties with lattices in $\Iso(\chn )$, 
and this alone makes $\pi_1(N)$ 
worth studying.
A key to understanding $\pi_1(N)$ lies in proving that
the group is relatively hyperbolic. The main technical result 
of this paper is

\begin{thm} \label{thm: main thm} If $M$ is a complete finite volume
complex hyperbolic $n$-manifold, 
and $S$ is a compact totally geodesic 
complex $(n-1)$-submanifold, then\newline
\textup{(i)} $M\setminus S$
admits a complete finite volume metric of $\sec\le -1$;\newline
\textup{(ii)} 
the group $\pi_1(N)$ is non-elementary (strongly) relatively hyperbolic, 
where the peripheral subgroups are fundamental groups of the 
components of $\d N$. \newline
\textup{(iii)} $M\setminus S$
admits a complete finite volume $A$-regular metric of $\sec<0$.
\end{thm}

The proof of Part (i) involves a delicate
warped product computation which occupies most of this paper,
and is sketched in Section~\ref{sec: sketch of curv comp}.
Using a special form of the metric constructed in (i), we prove,
following~\cite[Section 4]{Bel-rh-warp}, that $\pi_1(N)$ 
satisfies Gromov's definition of relative hyperbolicity 
elaborated in~\cite{Bow-rel}.  
It seems that (i) by itself yields no information 
about $\pi_1(N)$ beyond the obvious fact that
$\pi_1(N)$ surjects onto $\pi_1(M)$.

By contrast, (iii) has substantial topological implications, 
namely, Farrell-Jones \cite[Addendum 0.5]{FJ-A-reg} proved that 
the fundamental group of any complete manifold 
with $A$-regular metric of nonpositive curvature satisfies Borel's conjecture,
while Lafforgue~\cite[Corollary 0.0.4]{Laf-BC} proved Baum-Connes conjecture 
for the fundamental groups of complete $A$-regular nonpositively 
curved manifolds that satisfy Rapid Decay property
(cf. Corollary~\ref{thm: intro-gr-theoretic-cor}(10) below). 

Recall that a Riemannian metric is called {\it $A$-regular} 
if there exists a sequence of positive numbers $A=\{A_k\}$
such that for each $k\ge 0$,  
the $k$-th covariant derivative of the curvature tensor satisfies
$||\nabla^k R||_{C^0}<A_k$; for $k=0$ this
yields a two-sided bound on sectional curvature. 
Any metric on a compact manifold is $A$-regular, 
and similarly, this is true e.g. for an open manifold 
which is isometric outside a compact subset to the product
of a closed manifold and a ray. A locally symmetric metric is 
$A$-regular. 
A complete Riemannian 
metric with $a\le \sec\le b$ admits a $C^1$-nearby complete $A$-regular metric 
with almost the same curvature bounds (see~\cite{Kap-ricfl}), in particular,
if $a, b$ are negative, then the $A$-regular metric is negatively pinched;
however, if $b=0$, then the $A$-regular metric of~\cite{Kap-ricfl} need not 
be nonpositively curved. 
In general, proving that a complete nonpositively curved 
manifold admits a complete
$A$-regular nonpositively curved metric is quite difficult,
and in Theorem~\ref{thm: main thm}(i) this is done by  
modifying the metric constructed in (i), and proving $A$-regularity by
a brute force computation. 
Part (13) of Corollary~\ref{thm: intro-gr-theoretic-cor} 
below implies that $M\setminus S$ admits no complete
$A$-regular metric of $\sec\le -1$.

%
The following Mostow-type rigidity result is implied by part (ii) 
of Theorem~\ref{thm: main thm} combined with the classical 
Mostow-Prasad Rigidity.

\begin{cor}\label{thm: intro-mostow} 
For $n>1$ and $i=1,2$, suppose that
$M_i$ is a complete finite volume complex hyperbolic $n$-manifold,
and $S_i$ is a compact totally geodesic complex $(n-1)$-submanifold. Then
any homotopy equivalence $f\co M_1\setminus S_1\to M_2\setminus S_2$
give rise to an isometry $\iota_f\co M_1\to M_2$
taking $S_1$ to $S_2$ such that the restriction
$\iota_f\co M_1\setminus S_1\to M_2\setminus S_2$
is homotopic to $f$. Moreover, $\iota_f$ is uniquely
determined by the homotopy class of $f$.
\end{cor}

\begin{cor}\label{cor: intro-out}
For $n>1$, if
$M$ is a complete finite volume complex hyperbolic $n$-manifold,
and $S$ is a compact totally geodesic complex $(n-1)$-submanifold, 
then the correspondence $f\to \iota_f$ induces an isomorphism
of the outer automorphism group of $\pi_1(M\setminus S)$
onto the group of isometries of $M$ that map $S$ to itself. 
In particular, the outer automorphism group of $\pi_1(M\setminus S)$
is finite.
\end{cor}

Other corollaries of Theorem~\ref{thm: main thm} are summarized below.

\begin{cor}\label{thm: intro-gr-theoretic-cor} 
If $M$ is a complete finite volume complex hyperbolic $n$-manifold,
and $S$ is a compact totally geodesic complex $(n-1)$-submanifold
with $n>1$, then \newline
$\mathrm{(1)}$ the relatively hyperbolic boundary of $\pi_1(N)$ 
is the $(n-1)$-sphere.\newline
$\mathrm{(2)}$ 
$\pi_1(N)$ does not split as an amalgamated product
or an HNN-extension over subgroups of  
peripheral subgroups of $\pi_1(N)$, or over $\Z$.\newline
$\mathrm{(3)}$  $\pi_1(N)$ is co-Hopf.\newline
$\mathrm{(4)}$ for any finite subset $F\subset\pi_1(N)$  
there is a homomorphism
of $\pi_1(N)$ onto a non-elementary hyperbolic group  that is injective 
on $F$.\newline
$\mathrm{(5)}$ $\pi_1(N)$ satisfies Strong Tits Alternative
iff $M$ is compact.\newline
$\mathrm{(6)}$ $\pi_1(N)$ is biautomatic iff $M$ is compact.\newline
$\mathrm{(7)}$  No nontrivial subgroup of $\pi_1(N)$ has Kazhdan property 
\textup{(T)}. \newline
$\mathrm{(8)}$ $\pi_1(N)$ is not a $CAT(0)$ group.
\newline
$\mathrm{(9)}$ $\pi_1(N)$ has finite asymptotic dimension.\newline
$\mathrm{(10)}$ $\pi_1(N)$ has Rapid Decay property.\newline
$\mathrm{(11)}$ $\pi_1(N)$ satisfies the Baum-Connes conjecture.\newline
$\mathrm{(12)}$ $\pi_1(N)$ satisfies the Borel Conjecture, and in particular,
if $n>2$, then any homotopy equivalence 
of compact manifolds $L\to N$ that restricts to a homeomorphism
of the boundaries $\d L\to \d N$ is homotopic to a homeomorphism 
rel boundary.\newline
$\mathrm{(13)}$ $\pi_1(N)$ is not isomorphic to the fundamental group of a 
complete negatively pinched Riemannian manifold.\newline
$\mathrm{(14)}$ if $\pi_1(N)$ is isomorphic to a lattice $\Lambda$
in a real Lie group $G$, then the identity component $G_0$ of $G$ is compact,
$\Lambda\cap G_0$ is trivial, 
and $\Lambda$ projects isomorphically onto 
a finite index subgroup of $G/G_0$. \newline
$\mathrm{(15)}$ $\pi_1(N)$ has solvable word and conjugacy problems.
\end{cor}

It is instructive to see that
with exception of (8), (13), (14) 
of Corollary~\ref{thm: intro-gr-theoretic-cor},
all conclusions of Theorem~\ref{thm: main thm} and 
Corollaries~\ref{thm: intro-mostow}--\ref{thm: intro-gr-theoretic-cor},
are valid when $S=\emptyset$: indeed,
Corollaries~\ref{thm: intro-mostow}--\ref{thm: intro-gr-theoretic-cor}
still follow from Theorem~\ref{thm: main thm}, which becomes obvious for
$S=\emptyset$. Thus $M\setminus S$ with nonempty $S$
shares many properties of complete finite volume complex hyperbolic 
manifolds.

Corollaries~\ref{thm: intro-mostow}--\ref{thm: intro-gr-theoretic-cor}
follow by combining Theorem~\ref{thm: main thm} with
various works available in the literature, and 
with a few exceptions, their proofs are identical to the proofs of
the corresponding results in the real hyperbolic case studied
in~\cite{Bel-rh-warp}; the cases where the proofs are different
from those in~\cite{Bel-rh-warp} are dealt with in 
Section~\ref{sec: appl}. 

The only previous results on the topology and geometry of $M\setminus S$ 
are as follows. 
Toledo showed that the group $\pi_1(M\setminus S)$ 
is not K\"ahler 
when $n=2$~\cite[page 112]{ABCKT}, and is sometimes K\"ahler 
when $n>2$~\cite[pages 107-110]{Tol-nrf} (neither of these references 
treats the case of $M\setminus S$ directly, but the proofs still work
with minor modifications).
Allcock-Carlson-Toledo~\cite{ACT} studied a more general 
(and much more complicated) 
case when hyperplanes in $S$ are allowed to intersect orthogonally; 
they write down an explicit infinite presentation for
the kernel of the homomorphism
$\pi_1(M\setminus S)\to\pi_1(M)$ induced by the inclusion, and
prove that $\pi_1(M\setminus S)$
is not isomorphic to a lattice in a virtually connected real Lie group. 

I refer to~\cite{Bel-rh-warp} for some open problems about $M\setminus S$,
and focus on the following tantalizing question due to Toledo.

\begin{quest}\label{quest: rf}
Is $\pi_1(M\setminus S)$ residually finite?  
\end{quest}

Toledo~\cite{Tol-nrf} showed that the answer is generally {\it no} 
in the similar case when $(M,S)$ is modelled on $({\bf X}_n, {\bf X}_{n-1})$
where ${\bf X}_n$ is the symmetric space for $SO(n,2)$. 
It is instructive to recall his argument. Fix an arbitrary
component $B$ of the boundary of a small tubular neighborhood 
of $S$ in $M$. 
Toledo proves that the inclusion $B\to M\setminus S$ is 
$\pi_1$-injective and $\pi_1(B)$ is a lattice in the universal cover 
of $Spin(n-1, 2)$. The proof in~\cite{Tol-nrf}
is written under the simplifying assumptions that $n\ge 4$ 
and $n$ is even, in which case 
Raghunathan's work~\cite{Rag} implies that $\pi_1(B)$ 
is not residually finite.  
Since residual finiteness of a group is inherited by subgroups,
it followed that $\pi_1(M\setminus S)$ is not residually finite.
Toledo comments that the above proof also works when 
$(M,S)$ is modelled on $(\chn, \chm)$ except that~\cite{Rag}
is not available. Raghunathan's work is intimately related
with solution of the congruence subgroup problem for $SO(n-1,2)$, which is
wide open for lattices in $SU(n-1,1)$.  

If $(M,S)$ is modelled on $(\chn, \chm)$, then there is a strong 
motivation for trying to show that $\pi_1(M\setminus S)$ need not
be residually finite. Indeed, by Part (5) of 
Corollary~\ref{thm: intro-gr-theoretic-cor}, which is based on 
Dehn Surgery theorem in relatively hyperbolic groups
due to~\cite{Osi-periph, GroMan}, 
$\pi_1(M\setminus S)$ is {\it residually hyperbolic},
and therefore, {\it if} $\pi_1(M\setminus S)$ is
{\it not} residually finite for some $(M,S)$, then there exists
a hyperbolic group that is not residually finite. Of course, there are 
many residually hyperbolic groups that do not
look residually finite, and the main reason 
I single out $\pi_1(M\setminus S)$ as a candidate for disproving residual 
finiteness of hyperbolic groups
is that  $\pi_1(M\setminus S)$ is not far from being a lattice
so perhaps its finite index subgroups could be sometimes
understood via arithmetic means. 

Another promising candidate
is $\pi_1(B)$, where as before $B$ is the boundary of a tubular 
neighborhood of a component of $S$. Indeed, $\pi_1(B)$
is a lattice in the universal cover of $SU(n-1,1)$, 
which is a nonlinear semisimple Lie group 
(see e.g.~\cite[page 115]{ABCKT}). 
As noted in Lemma~\ref{lem: kahler form},
$B$ is a circle bundle over a component of $S$ whose
first Chern class is the $-\frac{1}{4\pi}$-multiple 
of the K\"ahler form of $S$, and therefore, $\pi_1(B)$ is an extension  
with infinite cyclic kernel and hyperbolic quotient, which
does not virtually split. 
By a straightforward argument,
any extension with infinite cyclic kernel and hyperbolic quotient
is residually hyperbolic, yet it is unclear whether
$\pi_1(B)$ is always residually finite. 

\section{Outline of the curvature computation}
\label{sec: sketch of curv comp}

In the context of this paper, a multiply-warped product is a metric 
of the form $dr^2+g_r$ where $r$ varies in an open interval
and $g_r$ is a family of Riemannian metrics on a smooth manifold
$F$ constructed by fixing a Riemannian metric $\mathbf f$ on $F$,
considering an orthogonal splitting of the tangent bundle
$TF$ into (possibly nonintegrable) subbundles $H_i$, and scaling 
the metric on each $H_i$ by a warping function $h_i=h_i(r)$.
The key issues in constructing multiply-warped metrics with 
prescribed curvature bounds are \newline
(1) to come up with curvature formulas   
such that the bounds on curvature translate into {\it simple} differential
inequalities on warping functions $h_i$,\newline
(2) to construct warping functions $h_i$ that satisfy the inequalities. 

Part (1) depends on the specifics of the geometry of $(F, \bf{f})$ 
and on interaction between $H_i$'s, e.g. the curvature formulas
for $(F, g_r)$ typically involve brackets of vector fields 
from different $H_i$'s, and if each $H_i$ is integrable, the 
formulas simplify considerably.
Part (2) is driven by the shape of the differential inequalities
obtained in Part (1). The methods used in Part (2)  
are usually those of single variable calculus and elementary ODE,
yet making them work is a specialized craft involving a number
of tricks, and the intuition behind the tricks is intimately 
related to the geometry of the desired curvature bound,
be that negative, almost nonnegative, or Ricci positive curvature.

In Section~\ref{sec: cyl-coord} we write the complex
hyperbolic metric on the ends of $M\setminus S$
in cylindrical coordinates about $S$
as
\[
dr^2+\sinh^2(r) d\theta^2 + \cosh^2\left(\frac{r}{2}\right) {\bf k}^{n-1}
\]
where $r$ is the distance to $S$, and $\theta$ is the parameter on
the unit circle about $S$, and ${\bf k}^{n-1}$ is the complex
hyperbolic metric $S$. The ``+'' refers to the orthogonal splitting of 
the tangent bundle to $M\setminus S$ into the sum of
integrable subbundles spanned by $\ddr$ and $\frac{\d}{\d\theta}$,
and their orthogonal complement $\cH$ , which is nonintegrable.
We then modify the metric on the ends of $M\setminus S$ to be
 \[
\l_{v,h}:=dr^2+v^2 d\theta^2 + h^2 {\bf k}^{n-1}
\]
and compute its curvature tensor in terms of $v$, $h$,
where $v, h$ are positive functions of $r$,
which varies from $-\infty$ to the normal injectivity radius of $S$.
Formulas of Appendix~\ref{sec: components-of-curv-tensor}
reduce the problem to computing curvatures of the 
$r$-tubes about $S$. In Section~\ref{sec: basis} we set up a convenient
frame in which the curvature tensor components are to be computed.
The ``structure constants'' of brackets in the frame
are computed in Section~\ref{sec: curv of chn} by specializing to the
complex hyperbolic space where all curvatures are known. 
Each $r$-tube about $S$
comes with the Riemannian submersion metric $v^2 d\theta^2 + h^2 {\bf k}^{n-1}$ 
which has totally geodesic circle fibers,  so we use 
O'Neill's formulas to compute the curvature tensor of the tube. 
This is done in Sections~\ref{sec: seccurv of coord planes}--\ref{sec: mixed}
where we also arrange for several computational simplifications, notably,
we shall never need to know $\langle R(X_i, X_j), X_k, X_1\rangle$
where $X_1$ is vertical and $X_i, X_j, X_k$ are linearly independent 
horizontal vector fields.
Putting all this together in Section~\ref{sec: sectional curv},
we obtain a reasonably simple formulas for the sectional curvature.

In Section~\ref{sec: curv estimate} we choose $v,h$ so that $M\setminus S$
becomes complete, finite volume, and of sectional curvature bounded
above by a negative constant, and furthermore the metric is complex
hyperbolic away from a small tubular neighborhood of $S$. 
This is the heart of the proof, and to help 
digesting it we outline what we shall do, and why we do it.

First of all, we assume that $v, h$ are positive 
so that the metric is nondegenerate.
After glancing over the curvature formulas 
(\ref{form: k(y_i,y_1)})--(\ref{form: r(dr,y_1, y_2, y_3)})
it is apparent that we need 
$h^{\prime\prime}$, $v^{\prime\prime}$ to be positive,
and furthermore, $h^\prime$, 
$v^\prime$ may not vanish for if $h^\prime v^\prime=0$, then 
$K(Y_2, Y_1)>0$. This means that
$v, h$ are increasing everywhere as they are equal to 
increasing functions $\sinh(r)$, $\cosh(r/2)$ for sufficiently large $r$.
As a starting point, we let $h(r)=e^{r/2}$ and
$v(r)=\e e^r$ on a neighborhood of $-\infty$, 
where $0<\e\ll 1$ is a parameter, 
and for these $h, v$ it is easy to compute 
that $\sec(\l_{v,h})<-\frac{1}{10}$. The main issue is to interpolate
$v, h$ in between while keeping curvature negative.

The graphs of $\sinh(r)$ and $\e e^r$ intersect at a point $r_\e\approx\e$, 
and  $v$ is obtained from the (strictly convex) function
$\max\{\sinh(r), \e e^r\}$ by smoothing it near $r_\e$ so that $v(r)=\e e^r$ 
for $r\le r_\e-2\e^4$ and $v(r)=\sinh(r)$ for for $r\ge r_\e+2\e^4$. 
While smoothing we need to be able to estimate
$\frac{v^\prime}{v}$, and also need to keep a definite lower bound on 
$\frac{v^{\prime\prime}}{v}$. This is accomplished by making
$v$ satisfy $(\ln(v))^{\prime\prime}>0$ over the smoothing interval, 
so that $\frac{v^{\prime\prime}}{v}>\left(\frac{v^\prime}{v}\right)^2$,
and $\frac{v^\prime}{v}$ is increasing, which allows us to
estimate $\frac{v^\prime}{v}$ by its values at the endpoints. 

Then we construct $h$ by bending down the graph of $\cosh(r/2)$ near 
$r=\e/2$ so that it eventually agrees with $e^{r/2}$. The tangent line
to the graph of $\cosh(r/2)$ near  $\e/2$
is almost horizontal, and it intersects
the graph of $e^{r/2}$ near $r=-\frac{8}{\e}$, and thus we  
bend $\cosh(\frac{r}{2})$ over the interval 
$[-\frac{8}{\e}, \frac{\e}{2}]$; note  that
on this interval $v=\e e^r$. 
Bending $h$ is done in two stages, which helps to control
$\frac{h^\prime}{h}$; during the first stage $h$ almost coincides
with the tangent line to $\cosh(r/2)$ at $\e/2$, and during
the second stage we bend $h$ upwards, so that $(\ln(h))^{\prime\prime}>0$.
At either stage we manage to estimate 
$\frac{h^{\prime\prime}}{h}$, $\frac{h^\prime}{h}$.

In fact, for technical reasons we build $v, h$ by
first producing ``easy-to-visualize''
$C^1$ functions $\mathbf v,\mathbf h$, 
which we then smooth via convolutions to 
get good lower bounds on the second derivatives of $v, h$
using Appendix~\ref{sec: appendix smooth convex}.

Finally, we estimate the curvature of $\l_{v,h}$ over two
disjoint intervals, one where $v$ is bent, and the other where
$h$ is bent. The main difficulty
is to control the ``mixed'' term (\ref{form: r(dr,y_1, y_2, y_3)}), 
and it turns out that the terms $K(Y_3, Y_2)$, $K(\ddr, Y_1)$, $K(\ddr, Y_2)$
in formulas (\ref{form: k(y_i,y_1)})--(\ref{form: k(dr,y_i)})
carry enough negative curvature to compensate the positivity
of the ``mixed'' term. 
Over the interval where $h$ is bent, 
$\frac{h^\prime}{h}$, $\frac{v^\prime}{v}$ are
kept bounded and $v=\e e^r$, so if $\e$  is small, then
$\frac{v}{h^2}$ becomes small when $\e\to 0$, and hence the ``mixed''
term is negligible. 
On the other hand, over the interval where $v$ is bent, 
the ``mixed'' term does not become small,
and instead it is compensated by $K(\ddr, Y_1)$ and 
$K(Y_3, Y_2)$, and the estimate hinges on 
how $c_{23}$ enters in $K(Y_3, Y_2)$ and in the ``mixed'' term.

The proof takes several pages of tedious curvature estimates, which
seems hard to shorten. Linear algebra arguments throughout this proof in 
repetitive, and  it is conceivable that they could be simplified 
by doing the computation in a different frame, 
e.g. the one that diagonalizes the curvature operator. Unfortunately, 
it seems that this would make the formulas for the sectional 
curvature of the coordinate planes much more complicated than those 
in formulas (\ref{form: k(y_i,y_1)})--(\ref{form: r(dr,y_1, y_2, y_3)}),
so at the end we would gain nothing.

After proving that sectional curvature is bounded above by a negative constant, 
we apply a result in~\cite{Bel-rh-warp} to check that 
$\pi_1(M\setminus S)$-action on the universal cover of $M\setminus S$
satisfies Gromov's definition of relative hyperbolicity, which proves
Part (i) of Theorem~\ref{thm: main thm}.

Part (ii) is proved in Section~\ref{sec: A-reg}.  
We keep $v=\e e^r$, and bend the function $h$ constructed 
above near $-\infty$ so that there it becomes equal to $\tau_\e+e^{r/2}$,
where $\tau_\e$ is a carefully chosen positive constant.
By formulas (\ref{form: k(y_i,y_1)})--(\ref{form: r(dr,y_1, y_2, y_3)}),
this choice of $h$ ensures that the sectional curvature
is bounded, and following the pattern of Part (i) 
we prove that the curvature is negative (but not bounded away from zero,
which would be impossible by Part (13) of 
Corollary~\ref{thm: intro-gr-theoretic-cor}). 
Furthermore, using
formulas (\ref{form: k(y_i,y_1)})--(\ref{form: r(dr,y_1, y_2, y_3)}), 
and the fact that $v=\e e^r$, $h=\tau_\e+e^{r/2}$ 
near $-\infty$, we are able to show that all the derivatives
of the curvature tensor have bounded components, so the metric is 
$A$-regular.

\section{Complex hyperbolic space in cylindrical coordinates}
\label{sec: cyl-coord}

We follow~\cite{Gol-book} for conventions and 
background on complex hyperbolic geometry. In particular,
the complex hyperbolic space $\chn$ is normalized to have
holomorphic sectional curvature $-1$, and we denote
the complex hyperbolic metric by ${\bf k}^n$, or simply by $\mathbf k$ 
for brevity.

The purpose of this section is to describe ``cylindrical
coordinates'' on $\chn$ about a complex hyperplane $\chm$. 
The boundary of the $r$-neighborhood of $\chm$ is a real 
hypersurface, which is denoted by $F(r)$, and is
referred to as an {\it $r$-tube}. Thus the metric on $\chn$
can then be written as ${\bf k}=dr^2+{\bf k}_r$ where ${\bf k}_r$ is the induced
Riemannian metric on $F(r)$, and we need to describe
${\bf k}_r$. This computation seems to be unknown to experts, 
so we give full details. 

The orthogonal projection $\pi\co\chn\to\chm$
is a fiber bundle whose fibers are complex geodesics, i.e.
totally geodesic complex submanifolds isometric to the
real hyperbolic plane of curvature $-1$~\cite[Theorem 3.1.9]{Gol-book}.
Restricting $\pi$ to $F(r)$ gives a circle bundle $\pi_r\co F(r)\to\chm$
whose fiber over $w\in\chm$ is the circle of radius $r$ in the complex 
geodesic $\pi^{-1}(w)$. The tangent bundle $TF(r)$ splits orthogonally
as $\cV(r)\oplus \cH(r)$, where $\cV(r)$ is tangent to the circle 
$\pi^{-1}(w)\cap F(r)$, and $\cH(r)$ is the orthogonal complement 
of $\cV(r)$. Thus any vector in $TF(r)$ can be uniquely decomposed as
$V+H$, where $V\in\cV(r)$ and $H\in \cH(r)$, and 
\[
{\bf k}_r(V+H,V+H)={\bf k}_r(V,V)+{\bf k}_r(H,H).
\] 
As we explain below
under suitable identifications,
the restriction of ${\bf k}_r$ to $\cV(r)$ is
$\sinh^2(r) d\theta^2$, and the restriction of ${\bf k}_r$ to 
$\cH(r)$ is $\cosh^2(\frac{r}{2}) {\bf k}^{n-1}$, 
where $d\theta^2$ is the standard metric on the unit
circle $\S1$.

As $\chn$ has nonpositive sectional curvature,
and $\chm$ is totally geodesic, the hyperplane
$\chm$ has infinite normal injectivity radius, and
the map $r\co\chn\setminus\chm\to (0,\infty)$ 
is a (smooth) Riemannian submersion with fibers $F(r)$.
The geodesic flow along  {\it radial} 
(i.e. orthogonal to $\chm$) geodesics 
induces a diffeomorphism between different tubes 
$\phi_{sr}\co F(s)\to F(r)$, $s, r>0$, and also 
preserves every complex geodesic orthogonal to $\chm$.
We are to prove later in this section that the differential
$d\phi_{sr}$ maps $H(s)$ to $H(r)$, and $V(s)$ to $V(r)$.
Fix an arbitrary radial unit speed geodesic $\g(r)$ 
with $\g(0)=w\in\chm$. 
 
That $d\phi_{sr}$ takes $\cV(s)$ to $\cV(r)$ is obvious 
because $\cV(r)$ is tangent both to $F(r)$ and the complex geodesic
$\pi^{-1}(w)$, thus $\phi_{sr}$ restricted to $\pi^{-1}(w)$
simply maps the $s$-circle centered at $w$ to the concentric $r$-circle,
whose tangent bundles are $\cV(s)$, $\cV(r)$, respectively. 
Since $\pi^{-1}(w)$ is a hyperbolic plane of 
curvature $-1$, its metric can be written as 
$dr^2+\sinh^2(r)d\theta^2$ where $d\theta^2$ is the standard
metric on the unit circle, so that the metric on 
$\cV(r)$ equals to $\sinh^2(r)d\theta^2$.
In the $(r,\theta)$-coordinates the map $\phi_{sr}$ becomes 
$(s,\theta)\to(r,\theta)$ because the lines $\theta=\mathrm{constant}$
are geodesics, and therefore, the vector field $\frac{\d}{\d\theta}$
is $d\phi_{sr}$-invariant. 

Let $\delta$ be a unit speed geodesic in $\chm$ with $\s(0)=w$.
By~\cite[Lemma 3.2.13]{Gol-book} the exponential map takes
the plane $\mathrm{span}(\g^\prime, \delta^\prime)\subset T_w\chm$ 
to a totally real (and hence totally geodesic)
$2$-plane $R_{\delta}$ in $\chn$ which intersects $\chm$ 
along $\delta$ and intersects the complex geodesic $\pi^{-1}(w)$
along $\g$.
If $t$ denotes the arclength parameter on $\delta$, then
the metric on $R_\delta$ can be written in the $(r,t)$-coordinates 
as $dr^2+q^2(r,t) dt^2$ where $r,t\in\mathbb R$, and in fact
$q(r,t)$ is independent of $t$ because the isometric 
$\mathbb R$-action on $R_\delta$ by translations along $\delta$
extends to an isometric $\mathbb R$-action on $\chn$. 
Since the sectional curvature of any totally real plane is 
$-1/4$~\cite[page 80]{Gol-book}, and since the sectional curvature 
of $dr^2+q^2(r) dt^2$ is $-\frac{q^{\prime\prime}}{q}$, we 
conclude that $q(r)=\cosh(\frac{r}{2})$.

Being totally geodesic, $R_{\delta}$ is preserved by the geodesic flow, 
and in the $(r,t)$-coordinates $\phi_{sr}$ maps $(s,t)$ to $(r,t)$ 
for $s,r\ge 0$. In particular,
the vector field $\frac{\d}{\d t}$ is $d\phi_{sr}$-invariant,
i.e. $d\phi_{sr}(\frac{\d}{\d t})=\frac{\d}{\d t}$.
It follows that $d\phi_{sr}(\frac{\d}{\d t})\in\cH$;
indeed $\frac{\d}{\d t}$ is clearly orthogonal to $\ddr$,
and $\frac{\d}{\d t}$ is orthogonal to 
$\frac{\d}{\d\theta}=J\ddr$ because the span of 
$\frac{\d}{\d t}$, $\ddr$ is totally real.

Since every vector in $T_w\chm$ is proportional to some 
$\s^\prime(0)=\frac{\d}{\d t}$, and
$\frac{\d}{\d t}$ is $d\phi_{sr}$-invariant,
the linear maps
$d\phi_{0r}\co T_w\chm\to \cH (r)$ and $d\phi_{sr}\co \cH(s)\to \cH (r)$,
are injective, and hence they must be isomorphisms by dimension reasons.

The length of $\frac{\d}{\d t}\in TR_\delta\subset\cH (r)$ 
is $\cosh(\frac{r}{2})$,
so the metric on $\cH(r)$ can be written as
$\cosh^2(\frac{r}{2}) {\bf k}^{n-1}$, or more explicitly, for $H\in\cH(r)$
\[
{\bf k}_r(H,H)=\cosh^2\left(\frac{r}{2}\right)
{\bf k}^{n-1}(d\phi_{0r}^{-1}(H),d\phi_{0r}^{-1}(H)).
\] 
Thus the metric $q_{r}$ on $TF(r)=\cV(r)\oplus\cH(r)$ 
can be written as 
$\sinh^2(r) d\theta^2 + \cosh^2(\frac{r}{2}) {\bf k}^{n-1}$. 
It is clear that $\sinh^2(r) d\theta^2 + \cosh^2(\frac{r}{2}){\bf k}^{n-1}$ 
is a Riemannian submersion metric whose base is the 
$\cosh(\frac{r}{2})$ multiple of $\chm$, and 
fibers are standard circles of radius $\sinh(r)$.

Finally, fix an arbitrary tube, denote it by $F$,
and use the diffeomorphisms $\phi_{sr}$ to pull back 
$\cV(r)$, $\cH(r)$, ${\bf k}_r$ to $F$. Since the pullbacks of 
$\cV(r)$, $\cH(r)$ are independent of $r$, we just denote 
the corresponding subbundles of $TF$ by $\cH$, $\cV$.
As $\pi\circ\phi_{sr}=\pi$,
the projections $\pi_r\co F(r)\to\chm$ all get identified
via $\phi_{sr}$ to a circle bundle projection $F\to\chm$ 
whose differential takes $\cV$ to zero, and maps $\cH$ onto $T\chm$.
In summary, the complex hyperbolic manifold
$\chn\setminus\chm$ is now written as $(0,\infty )\times F$ equipped 
with the metric 
\[
dr^2+\sinh^2(r) d\theta^2 + \cosh^2\left(\frac{r}{2}\right) {\bf k}^{n-1}.
\]

\section{Basis and brackets}
\label{sec: basis}

We borrow the notations $F, \cH,\cV, k_n$ from 
Section~\ref{sec: cyl-coord}, and fix an open interval $I$.
Given positive smooth functions $v$, $h$ on $I$, 
let $\l_{v,h,r}$ be the Riemannian submersion metric 
on $TF=\cV\oplus\cH$ with base $h\chm$ and fiber $v\S1$; 
we also write 
\[
\l_{v,h,r}:=v^2 d\theta^2 + h^2 {\bf k}^{n-1} .
\] 
This gives rise to the metric 
$\l_{v,h}=dr^2+\l_{v,h,r}$ on $I\times F$. For brevity
we sometimes suppress $v,h$ and label tensors associated with
$\l_{v,h}$, $\l_{v,h,r}$ by $\l$, $\l_r$, respectively.

\begin{ex}
If $I=(0,\infty )$, $v=\sinh(r)$ and $h=\cosh(\frac{r}{2})$, 
then  $\l_{v,h,r}={\bf k}_r$ so that $\l_{v,h}=dr^2+{\bf k}_r={\bf k}^n$
is the complex hyperbolic metric. 
\end{ex}

The purpose of this section is to introduce a convenient local orthonormal
frame on $I\times F$ in which the curvature of 
$\l_{v,h}$ will be computed.
To this end denote $\frac{\d}{\d r}$ by $\ddr$, and 
$\frac{\d}{\d\theta}$ by $X_1$.
Fix $z\in I\times F$ and let $w\in\chm$ be the image
of $z$ under the map $p\co I\times F\to\chm$ 
obtained by composing the projection to the second
factor $I\times F\to F$ with the circle bundle 
$F\to \chm$.

Let $\{\check X_i\}$, with $1<i<2n$, 
be an arbitrary orthonormal frame defined on a 
neighborhood of $w$ in $\chm$ such that 
$[\check X_i,\check X_j]$ vanishes at $w$.
(By a standard argument any orthonormal basis in $T_w\chm$ can be extended
to some $\{\check X_i\}$ as above).   
Let $X_i$ be the vector field obtained by lifting $\check X_i$ to
a horizontal vector field in $\cH\subset TF$, 
and then pulling it back via the projection $I\times F\to F$.
Then $\ddr, X_1,\dots , X_{2n-1}$ is an orthogonal frame near $z$ such that
for all $i,j=1,\dots, 2n-1$ we have:
\begin{itemize}
\item [\rm(1)] $\langle X_1,X_1\rangle_\l=v^2$, 
and $\langle X_i, X_i\rangle_\l =h^2$ for $i>1$. 
\item [\rm(2)] $[X_i, X_j]$ is tangent to level surfaces of $r$,
\item  [\rm(3)] $[X_i, \ddr ]=0$  because 
$X_i$ is invariant under the flow of $\ddr$,
\item [\rm(4)] 
$[X_i, X_1]=0$ because $X_i$ is invariant under the flow of $X_1$ on $F$
that corresponds to rotation about $\chm$ in $\chn$. 
\item [\rm(5)] $[X_i, X_j]$ is vertical at $z$ because 
$[\check X_i,\check X_j]$ vanishes at $w$. 
\end{itemize}
By (4)-(5) there are ``structure constants''
$c_{ij}\in\mathbb R$ with $[X_i, X_j]=c_{ij}X_1$  at $z$.
Note that $c_{ij}=-c_{ji}$ and $c_{1j}=0$ for all $i,j$,
yet $c_{ij}$ generally depend on $z$ and $\{X_i\}$.
In Section~\ref{sec: curv of chn} below we derive some
estimates and identities involving $c_{ij}$'s.

The corresponding orthonormal frame $\ddr$, 
$Y_1=\frac{1}{v}X_1$, $Y_i=\frac{1}{h}X_i$, $i>1$ 
enjoys the following properties:
 \begin{itemize}
\item[\rm(i)] $[Y_i, Y_j]=\frac{1}{h^2}[X_i, X_j]=
c_{ij}\frac{v}{h^2}Y_1$ for $i,j>1$,
\item[\rm(ii)] $[Y_i, Y_1]=\frac{1}{hv}[X_i, X_1]=0$,
\item[\rm(iii)]  $[Y_1,\ddr ]=\frac{v^\prime}{v}Y_i$, 
and $[Y_i,\ddr ]=\frac{h^\prime}{h}Y_i$ for $i>1$,
\end{itemize}
where the first equalities in (i), (ii) hold because 
any function of $r$ has zero derivative in the direction of $X_i$, 
and (iii) follows from $[X_i, \ddr ]=0$.

\section{A curvature formula in the complex hyperbolic space}
\label{sec: curv of chn}

In Appendix~\ref{sec: components-of-curv-tensor} we explain, 
following~\cite{BW, Bel-rh-warp},
how to relate the components of the 
curvature tensors of $\l_{v,h}$ and $\l_{v,h,r}$ 
in the basis $\{\ddr, Y_1, \dots ,Y_{2n-1}\}$, but more work is needed to 
compute these components in terms of $v,h$ only, and this section provides one 
of the main steps in the computation. 
Specifically, we compute the components of the curvature tensor 
of ${\bf k}^n$ in the basis $\{\ddr, Y_1, \dots ,Y_{2n-1}\}$, and 
establish some useful identities on the structure constants $c_{ij}$.

{\bf Convention:}\it\ 
In this section $\langle\cdot,\cdot\rangle$, $R$, $J$  denote the metric, 
the curvature tensor, and the complex structure on $\chn$, respectively; 
in other words, we suppress the subscript $\mathbf k$ for all tensors throughout 
the section.\rm

In the complex hyperbolic space one has the following explicit formula 
for the $(4,0)$-curvature tensor $R$ in terms of $\bf k$ and $J$
(see~\cite[Proposition IX.7.3]{KN-II}): 
for any tangent vectors $X$, $Y$, $Z$, $W$
\begin{eqnarray*}
& 4\langle R (X,Y)Z,W\rangle=
\langle X, Z\rangle  \langle Y, W\rangle - 
\langle X, W\rangle \langle Y, Z\rangle +\\
& \langle X, JZ\rangle \langle Y, JW\rangle -
\langle X, JW\rangle \langle Y, JZ\rangle +
2\langle X,JY\rangle\langle Z, JW\rangle.
\end{eqnarray*}
If $X$, $Y$, $Z$, $W$ are vector fields in 
the $\bf k$-orthonormal basis
$\{\ddr, Y_1, \dots Y_{2n-1}\}$ with 
$X_1=Y_1\sinh(r)$ and $X_i=Y_i\cosh(\frac{r}{2})$ for $i>1$, 
then
\begin{eqnarray}
& \label{formula: mixed chn with dr}
\langle R(\ddr ,Y_1) Y_i,Y_j\rangle=
\frac{1}{2}\langle\ddr , JY_1\rangle\langle Y_i ,JY_j\rangle=
-\frac{1}{2}\langle Y_i ,JY_j\rangle\ \ \mathrm{if}\ \ i,j>1,\\
& \label{formula: mixed chn without dr}
\langle R (Y_i,Y_j)Y_j,Y_k\rangle =0\ \ 
\mathrm{if}\ \ i,j,k\ \ \mathrm{are\ distinct\ and}\ \ 1\in\{i,j,k\},\\
& \label{formula: sec cuv of chn}
\langle R (Y_i,Y_j)Y_j,Y_i\rangle=
-\frac{1}{4} - \frac{3}{4}\langle Y_i,JY_j\rangle^2
\ \ \mathrm{if}\ \ i\neq j.
\end{eqnarray}
In computing (\ref{formula: mixed chn with dr}) 
the first two summands of $4\langle R (X,Y)Z,W\rangle$
vanish because $\{\ddr, Y_1, \dots Y_{2n-1}\}$
is orthonormal, and the next two summands vanish because
$\langle\ddr , JY_k\rangle =0$ unless $k=1$.
The last equality in  (\ref{formula: mixed chn with dr}) is true because
$\langle\ddr , JY_1\rangle=-1$ with respect to the standard complex structure. 

In computing (\ref{formula: mixed chn without dr}), 
the first two summands vanish as $i,j,k$ are distinct and 
$\{Y_i\}$ is orthonormal.
The fourth summand vanishes because 
$Y_j$, $JY_j$ are orthogonal. Finally, 
$\langle Y_i,JY_j\rangle\langle Y_j, JY_k\rangle =0$ because 
$\cH$ is $J$-invariant and orthogonal to $\cV$,
and $1\in\{i,j,k\}$.

In computing (\ref{formula: sec cuv of chn}), 
the first summand vanishes and the second summand is $-1$
because $\{Y_i\}$ is orthonormal. The fourth summand vanishes
as $\langle Y_i,JY_i\rangle = 0$, and the other two summands 
may survive and add up to the promised expression since 
$\langle Y_i,JY_j\rangle=-\langle JY_i,Y_j\rangle$.

\begin{lem}\label{lem: c_ij via inner product}
$\langle Y_i,JY_j\rangle=2c_{ij}$.
\end{lem}
\begin{proof} 
The result follows by combining (\ref{formula: mixed chn with dr}) with
\begin{eqnarray}
& \langle R (\ddr ,Y_1) Y_i,Y_j\rangle=
\langle [Y_j,Y_i],Y_1\rangle
\left(\ln\frac{v}{h}\right)^\prime =
-c_{ij}\frac{v}{h^2}\left(\ln\frac{v}{h}\right)^\prime=-c_{ij},
\end{eqnarray}
where the first equality comes 
from Appendix~\ref{sec: components-of-curv-tensor} and $[Y_i,Y_1]=0$,
the second equality follows from 
$[Y_i, Y_j]=c_{ij}\frac{v}{h^2}Y_1$, and the last equality holds 
by explicit computation with $v=\sinh(r)$ and $h=\cosh(\frac{r}{2})$. 
\end{proof}

\begin{rmk}
As $|\langle Y_i,JY_j\rangle|\le 1$, 
Lemma~\ref{lem: c_ij via inner product}, and 
formulas (\ref{formula: mixed chn with dr}), 
(\ref{formula: sec cuv of chn}) imply
\begin{eqnarray}
& \label{formula: bound on c_ij}
|\langle R(\ddr ,Y_1) Y_i,Y_j\rangle|=|c_{ij}|\le\frac{1}{2},\\
& \label{formula: sec cuv of chn via c_ij}
\langle R(Y_i,Y_j)Y_j,Y_i\rangle=
-\frac{1}{4} - 3c_{ij}^2\in [-1,-\frac{1}{4}]
\ \ \mathrm{for}\ \ i\neq j.
\end{eqnarray}
\end{rmk}

\begin{lem}\label{lem: sec curv on base}
$\langle R (\check X_i,\check X_j)\check X_j,\check X_i\rangle=
-\frac{1}{4} - 3c_{ij}^2$.
\end{lem}
\begin{proof} For $i>1$,
the vector field $X_i$  is the horizontal lift of $\check X_i$
under the orthogonal projection of $\chn\to\chm$,
so on $\chm$ we have $\check X_i=X_i=Y_i$ as $\cosh(0)=1$. 
By continuity 
the formula (\ref{formula: sec cuv of chn via c_ij}) still holds 
for $r=0$ giving the promised result.
\end{proof}

\begin{rmk}
Since $\cH$ is $J$-invariant, $JY_i\in \cH$ for $i>1$, so
writing $Y_i$ in the orthonormal basis 
$\{Y_j\}$, $j>1$ of $\cH$, we get
\[
JY_i=\sum_{j>1}\langle JY_i , Y_j\rangle Y_j=-2\sum_{j>1} c_{ij} Y_j,
\] 
and since $|JY_i|=1$, we obtain the identity
\begin{eqnarray}
\label{formula: sum of c_ij}
\sum_{j>1} c_{ij}^2=\frac{1}{4}.
\end{eqnarray}
\end{rmk}

The following lemma and subsequent examples are not used
elsewhere in the paper, yet they may help understand $c_{ij}$'s
via complex hyperbolic geometry.

\begin{lem} If $i\neq j$, then $c_{ij}=0$ iff $Y_i, Y_j$ 
span a totally real plane. 
\end{lem}
\begin{proof} A plane is called {\it totally real}
if it is orthogonal to its $J$-image. Since 
$Y_i$ is always orthogonal to $JY_i$, the span of
$Y_i, Y_j$ is orthogonal to its image iff $Y_i$ and
$JY_j$ are orthogonal, which
by Lemma~\ref{lem: c_ij via inner product} is equivalent to
$c_{ij}=0$. 
\end{proof}

\begin{ex}
By (ii) of Section~\ref{sec: basis}, $c_{i1}=0$; alternatively 
$\langle Y_i,JY_1\rangle = 0$ as $\cH$ is $J$-invariant, 
and $\langle Y_1,JY_1\rangle = 0$. 
So (\ref{formula: sec cuv of chn}) implies that
$\langle R (Y_i,Y_1)Y_1,Y_i\rangle=
-\frac{1}{4}$ for $i>1$, which by~\cite[Lemma 3.2.13]{Gol-book}
also follows from the fact that $\{Y_1, Y_i\}$ 
spans a totally real plane.
\end{ex}

\begin{ex} By Lemma~\ref{lem: c_ij via inner product},
$|c_{ij}|=\frac{1}{2}$ iff $Y_i=\pm JY_j$, which in turn is
equivalent to saying that
$\{Y_i, Y_j\}$ spans a complex geodesic, i.e.
a totally geodesic complex line. 
Note that complex geodesics
have sectional curvature $-1$. 
If $n=2$, then $\cH$ has complex dimension one, 
which forces $Y_3=\pm JY_2$
so that  $c_{23}=\pm \frac{1}{2}$. 
\end{ex}

\section{Computing $A$ and $T$ tensors}
\label{sec: A and T}

In this section we compute the $A$ and $T$ tensors of the metric
\[
q_{v,h,r}:=\frac{1}{h^2}\l_{v,h,r}=\frac{v^2}{h^2}d\theta^2+{\bf k}^{n-1},
\]
which is a Riemannian submersion metric with base $\chm$ 
and fiber $\frac{v}{h} \S1$. 
We often suppress $v,h$ and denote $q_{v,h,r}$ by $q_r$,
and the associated norm by $|\cdot |_{q_r}$. 

First show that the fibers of the Riemannian submersion
are totally geodesic i.e. $T=0$. Indeed,
rescaling takes geodesics to geodesics, so 
it suffices to give a proof for $\l_{v,h,r}$. 
Then we need to show that $\nabla_{Y_1} Y_1$ is proportional to $\ddr$,
which is obvious because $\nabla_{Y_1} Y_1$ is tangent to 
the complex geodesic orthogonal to $\chm$, and $\nabla_{Y_1} Y_1$ is orthogonal to $Y_1$
because $Y_1$ has $\l_r$-length one.

Next compute the $A$ tensor of $q_r$ 
at the point $z$ in the basis $\{X_i\}$ chosen 
in Section~\ref{sec: basis}.
If $i,j>1$, then $A_{X_i} X_j$ is the $q_r$-orthogonal projection 
of $\frac{1}{2}[X_i,X_j]=\frac{1}{2} c_{ij}X_1$
to $\cV$, hence $A_{X_i} X_j=\frac{c_{ij}}{2} X_1$. 
Since $|X_1|_{q_r}=\frac{v}{h}$,  
we get for $i,j>1$ that $|A_{X_i}X_j|_{q_r}=
\frac{|c_{ij}|}{2}\frac{v}{h}$, and in particular,
$|A_{X_i}X_j|_{q_r}\le\frac{v}{4h}$. 
By~\cite[9.21d]{Bes} for $i>1$ we have
\[
\langle A_{X_i} X_1, X_j\rangle_{q_r}=
-\langle A_{X_i} X_j, X_1\rangle_{q_r}=-\frac{v^2}{h^2}\frac{c_{ij}}{2},
\] 
%
%
so since $\{X_i\}$ is 
a $q_r$-orthonormal basis in $\cH$ and $A_{X_i} X_1$ is horizontal, 
we get
\[
A_{X_i} X_1=-\frac{v^2}{2h^2}\sum_{j>1} c_{ij}X_j,\text{\   so that\   }
|A_{X_i} X_1|_{q_r}=\frac{v^2}{2h^2}\sqrt{\sum_{j>1} c_{ij}^2}=
\frac{v^2}{4h^2}
\]
where the second equality in the latter formula follows from 
(\ref{formula: sum of c_ij}).
With the $A$ tensor is computed, O'Neill's 
formulas~\cite[Theorem 9.28]{Bes} allow us
to calculate the sectional curvature: for distinct $i,j>1$ we get
\begin{eqnarray}
&  \label{formula: sec curv in X_i, X_1}
\langle R_{q_r}(X_1 ,X_i) X_i, X_1\rangle_{q_r}=
|A_{X_i} X_1|^2_{q_r}=\frac{v^4}{16h^4}\\
&  \label{formula: sec curv in X_i, X_j}
\langle R_{q_r}(X_i ,X_j) X_j,X_i\rangle_{q_r}=
-\frac{1}{4} - 3c_{ij}^2- 3c_{ij}^2\frac{v^2}{4h^2}
\end{eqnarray}
where (\ref{formula: sec curv in X_i, X_j}) depends on
Lemma~\ref{lem: sec curv on base}.

\section{Sectional curvatures of coordinate planes}
\label{sec: seccurv of coord planes}

Since the $(4,0)$-curvature tensor scales like the metric,
we get for $i>1$
\begin{eqnarray*}
& \langle R_{\l_r}(Y_1 ,Y_i) Y_i,Y_1\rangle_{\l_r}=
\frac{h^2}{v^2h^2}\langle R_{q_r}(X_1 , X_i) X_i, X_1\rangle_{q_r}
\end{eqnarray*}
and hence formulas (\ref{formula: sec curv in X_i, X_1}),
(\ref{form: curv of warped prod})  imply
\begin{eqnarray}\label{form: lambd-curv of 1i-plane}
& \langle R_{\l}(Y_1 ,Y_i) Y_i,Y_1\rangle_{\l}=
\frac{v^2}{16h^4}-\frac{v^\prime}{v}\frac{h^\prime}{h}.
\end{eqnarray}
Similarly, for distinct $i,j>1$ 
\begin{eqnarray*}
& \langle R_{\l_r}(Y_i ,Y_j) Y_j,Y_i\rangle_{\l_r}=
\frac{h^2}{h^4}\langle R_{q_r}(X_i , X_j) X_j, X_i\rangle_{q_r}
\end{eqnarray*}
so formulas (\ref{formula: sec curv in X_i, X_j}),
(\ref{form: curv of warped prod}) imply  
\begin{eqnarray}\label{form: lambd-curv of ij-plane}
& \langle R_{\l}(Y_i ,Y_j) Y_j,Y_i\rangle_{\l}=
-\frac{1}{4h^2} - \frac{3}{h^2}c_{ij}^2- 
3c_{ij}^2\frac{v^2}{4h^4}-\left(\frac{h^\prime}{h}\right)^2.
\end{eqnarray}
and also by (\ref{form: curv of warped prod})
\begin{eqnarray}\label{form: lambd-curv of ir-plane}
& \langle R_\l(Y_i,\ddr)\ddr ), Y_i \rangle_\l=
-\frac{h^{\prime\prime}}{h},
\ \ \ \ \ 
\langle R_\l(Y_1,\ddr)\ddr ), Y_1 \rangle_\l=
-\frac{v^{\prime\prime}}{v}.
\end{eqnarray}

\begin{rmk} A wary reader may wish
to play with trigonometric identities to verify the formulas 
(\ref{form: lambd-curv of 1i-plane}),
(\ref{form: lambd-curv of ij-plane}), 
(\ref{form: lambd-curv of ir-plane})
in the complex hyperbolic case, where for $i>j\ge 1$
the planes spanned by $\{Y_i,\ddr\}$,   or by $\{Y_i, Y_j\}$ with $c_{ij}=0$
are always totally real, and hence 
their sectional curvature is $-\frac{1}{4}$, while
$\{Y_1,\ddr\}$, or $\{Y_i, Y_j\}$ with $c_{ij}=\pm\frac{1}{2}$
span complex geodesics of curvature $-1$.
\end{rmk}

\section{Mixed components of the curvature tensor}
\label{sec: mixed}

As we show in Section~\ref{sec: sectional curv}, in order 
to compute the sectional curvatures of $\l$, one only
needs to know the components of $R_\l$ involving
$\ddr, Y_1, Y_2, Y_3$. In this section we show that 
all the mixed components involving
$\ddr, Y_1, Y_2, Y_3$ vanish except for those
listed (up to symmetries of the curvature tensor)
in (\ref{form: mixed1}), (\ref{form: mixed2}) below.
A component is called {\it mixed\  \!} if it involves $>2$
distinct basis vectors.

First, note that by Appendix~\ref{sec: components-of-curv-tensor} we have
$\langle R_g(Y_i,\ddr)\ddr ), Y_j \rangle=0$ if $i\neq j$, and
since $[Y_i, Y_j]=c_{ij}\frac{v}{h^2}Y_1$ and 
$c_{i1}=0=c_{1j}$,  the only terms of the form 
$\langle R_{\l}(\ddr ,Y_i) Y_j,Y_k\rangle_\l$ 
that could be nonzero at $z$ are 
as follows (up to symmetries of the curvature tensor):
\begin{eqnarray}
& \label{form: mixed1}
\langle R_{\l}(\ddr ,Y_1) Y_i,Y_j\rangle_\l=
\langle [Y_j,Y_i],Y_1\rangle_\l
\left(\ln\frac{v}{h}\right)^\prime =
-c_{ij}\frac{v}{h^2}\left(\ln\frac{v}{h}\right)^\prime ,\\
& \label{form: mixed2}
2\langle R_\l(\ddr ,Y_i) Y_j,Y_1\rangle_\l=
\langle [Y_i,Y_j],Y_1\rangle_\l 
\left(\ln\frac{v}{h}\right)^\prime =
c_{ij}\frac{v}{h^2}\left(\ln\frac{v}{h}\right)^\prime, 
\end{eqnarray}
where $i,j>1$ and $i\neq j$.
The remaining mixed components involving
$Y_1, Y_2, Y_3$ vanish by the following.
\begin{lem} \label{lem: zero-mixed-curv} 
$\langle R_{\l} (Y_i,Y_j)Y_j,Y_k\rangle =0$
if $i,j,k$ are distinct and $1\in\{i,j,k\}$.
\end{lem}
\begin{proof}
The idea of the proof is to show that
$\langle R_\l (Y_i,Y_j)Y_j,Y_k\rangle_\l$ is proportional
to $\langle R_k (Y_i,Y_j)Y_j,Y_k\rangle_k$ which is
zero by (\ref{formula: mixed chn without dr}).
By the formula (\ref{form: curv of warped prod}) and the fact
that the curvature tensor scales like the metric we have
\[
\langle R_\l (Y_i,Y_j)Y_j,Y_k\rangle_\l=
\langle R_{\l_r} (Y_i,Y_j)Y_j,Y_k\rangle_{\l_r}=
\frac{1}{h^2}\langle R_{q_r} (Y_i,Y_j)Y_j,Y_k\rangle.
\]
As shown in Section~\ref{sec: A and T},
$q_r$ is a Riemannian submersion metric with 
base $\chm$ and totally geodesic fiber $\frac{v}{h} \S1$.
If $v=\sinh(r)$ and $h=\cosh(\frac{r}{2})$, we denote 
$q_r$ by $q^{sc}_{r}$. Fix an arbitrary $r>0$, and let 
$t$ be the positive number satisfying
\[
\sqrt{t}\frac{\sinh(r)}{\cosh(\frac{r}{2})}=\frac{v}{h},
\] 
so that $q_r$ is obtained from
the Riemannian submersion metric $q_r^{sc}$
by rescaling the fiber by $t$.
The curvature tensors of $q_r$, $q_r^{sc}$
are related by
O'Neill's formulas~\cite[Theorem 9.28ce and Lemma 9.69ac]{Bes}
via $T$ and $A$ tensors of the submersion, 
and as we show below
$q_r$, $q_r^{sc}$ have proportional $ijjk$-components of the curvature
tensor, which finishes the proof. 

It remains to show proportionality of $ijjk$-components. 
The Riemannian submersion metric $q_r$
satisfies $T=0$, as the fibers are totally geodesic.
Also $\langle (\nabla_{Y_1}A)_{Y_j}Y_k,Y_1\rangle_{q_r}=0$ 
for distinct $j,k>1$, as follows e.g. from the last identity 
in~\cite[9.32]{Bes}; in essence this term vanishes because the fiber
is one-dimensional and $T=0$.

So by~\cite[Theorem 9.28c]{Bes}
\[
\langle R_{q_r} (Y_i,Y_1)Y_1,Y_k\rangle_{q_r}=
-\langle A_{Y_i} Y_1, A_{Y_k} Y_1\rangle_{q_r}
\]
where by~\cite[Lemma 9.69a]{Bes} the right hand side is the 
$t^2$-multiple of the same quantity for $q_{sc}$,
which equals to $\langle R_{q_r^{sc}} (Y_i,Y_1)Y_1,Y_k\rangle_{q^{sc}_r}=0$.

Similarly, by~\cite[Theorem 9.28e]{Bes},
\[
\langle R_{q_r} (Y_i,Y_j)Y_j,Y_1\rangle_{q_r}=
\langle (\nabla_{Y_j}A)_{Y_i} Y_j, Y_1\rangle_{q_r}
\]
where by~\cite[Lemma 9.69c]{Bes} the right hand side is the $t$-multiple
of the same quantity for $q_r^{sc}$, which
equals to $\langle R_{q_r^{sc}} (Y_i,Y_j)Y_j,Y_1\rangle_{q^{sc}_r}=0$.
The case when $Y_1$ occupies the first slot, instead of the last, 
follows from the symmetry of the curvature tensor.
\end{proof}

\section{Sectional curvature}
\label{sec: sectional curv}

In this section we compute the sectional curvature of $\l_{v,h}$
in terms of $v,h$. 
Fix an arbitrary $2$-plane $\s$ that is tangent to 
$I\times F$ at the point $z\in \{r\}\times F$.
As in Section~\ref{sec: basis}, we denote the projection 
$I\times F\to\chm$ by $p$, and let $w=p(z)\in\chm$.
  
We first focus on the ``generic'' case when the subspace 
$dp (\s)\subset T_w\chm$ is $2$-dimensional, and treat 
the case of $\dim(dp (\s))<2$ in Remark~\ref{rmk: non-generic sigma}.

To simplify the computation we choose 
a frame $\{Y_i\}$ depending on the position of $\s$.
Since $\{r\}\times F\subset I\times F$
has codimension one, $\s$ contains 
a unit vector $D$ that is tangent to $\{r\}\times F$.
Let $H_2\in dp(\s)$ be a unit vector proportional to 
$dp(D)$, and let $H_3\in dp(\s)$ be a vector such that
$\{H_2, H_3\}$ is an orthonormal basis of $dp(\s)$.
As in Section~\ref{sec: basis}, we extend $\{H_2, H_3\}$
to the frame $\{\check X_2,\dots, \check X_{2n-1}\}$ in $\chm$
satisfying $\check X_2=H_2$, $\check X_3=H_3$ at $w$, and then lift 
each $\check X_i$ to a horizontal vector field $X_i$, 
so that $Y_i=X_i/h$ is the corresponding unit vector field.
Thus $\ddr, Y_1, Y_2,\dots, Y_{2n-1}$ is a local frame near $z$.

Since $D$, $Y_2$ are tangent to $\{r\}\times F$,
and $dp(D)$ is proportional to $dp(Y_2)$, we conclude that
$D$ lies in the span of $Y_1, Y_2$. Let $C\in\s$
be a unit vector which is orthogonal to $D$. By construction
$dp(C)\subset dp(\s)$ lies in the span of $dp(Y_2)$, $dp(Y_3)$,
so $C$ lies in the span of $\ddr$, $Y_1$, $Y_2$, $Y_3$.
Thus $\{C,D\}$ is an orthonormal basis in $\s$ such that
\[
C=c_0\ddr+c_1Y_1+c_2Y_2+c_3Y_3, \quad D=d_1Y_1+d_2Y_2,
\]
for some $c_i, d_j\in\mathbb R$. 

For brevity, in this section
we suppress the subscript $\l$ in the metric and curvature
tensors of $\l$, and also denote by $K$ the sectional curvature
of $\l$. Symmetries of the curvature tensor, 
and Section~\ref{sec: mixed} imply the following.
\begin{eqnarray*}
& K(C , D)=
d_1^2\langle R(C , Y_1) Y_1, C\rangle 
+d_2^2\langle R(C , Y_2) Y_2, C\rangle 
+2d_1d_2\langle R(C , Y_1) Y_2, C\rangle\\
& \langle R(C , Y_1) Y_1, C\rangle=
c_2^2 K (Y_2,Y_1)+c_3^2 K (Y_3,Y_1) + c_0^2 K (\ddr,Y_1),\\
& \langle R(C , Y_2) Y_2, C\rangle=
c_1^2 K(Y_2,Y_1)+c_3^2 K (Y_3,Y_2) + c_0^2 K (\ddr,Y_2),\\
& \langle R(C , Y_1) Y_2, C\rangle=
-c_1c_2 K (Y_2,Y_1)+
\frac{3}{2}c_0c_3 \langle R(\ddr , Y_1) Y_2, Y_3\rangle, 
\end{eqnarray*}
where by Section~\ref{sec: mixed}
all but two mixed terms vanish, and the nonzero mixed
terms add up to
$\frac{3}{2}c_0c_3 \langle R(\ddr , Y_1) Y_2, Y_3\rangle$.
Thus $K(C , D)$ equals to
\begin{eqnarray}
\label{form: k(c,d)}
& (d_1c_2-d_2c_1)^2 K (Y_2,Y_1)+
d_1^2c_3^2 K(Y_3,Y_1)+d_1^2c_0^2 K(\ddr,Y_1)+\\
\nonumber & 
d_2^2c_0^2 K(\ddr,Y_2)+
d_2^2c_3^2 K(Y_3,Y_2)+
3d_1d_2c_0c_3 \langle R(\ddr , Y_1) Y_2, Y_3\rangle.
\end{eqnarray}
It follows from
(\ref{form: lambd-curv of 1i-plane})--(\ref{form: lambd-curv of ir-plane})
and (\ref{form: mixed1})--(\ref{form: mixed2}) that
\begin{eqnarray} 
& \label{form: k(y_i,y_1)}
K (Y_2,Y_1)=K(Y_3,Y_1)=
\frac{v^2}{16h^4}-\frac{v^\prime}{v}\frac{h^\prime}{h},\\
& \label{form: k(y_3,y_2)}
K (Y_3,Y_2)=-\frac{1}{4h^2} - \frac{3}{h^2}c_{23}^2- 
3c_{23}^2\frac{v^2}{4h^4}-\left(\frac{h^\prime}{h}\right)^2,\\
& \label{form: k(dr,y_i)}
K(\ddr,Y_1)=-\frac{v^{\prime\prime}}{v}, \ \ \ \ \
K(\ddr,Y_2)=-\frac{h^{\prime\prime}}{h},\\
& \label{form: r(dr,y_1, y_2, y_3)}
\langle R(\ddr , Y_1) Y_2, Y_3\rangle=
-c_{23}\frac{v}{h^2}\left(\ln\frac{v}{h}\right)^\prime=
-c_{23}\frac{v}{h^2}\left(\frac{v^\prime}{v}-
\frac{h^\prime}{h}\right).
\end{eqnarray}
where $|c_{23}|\le\frac{1}{2}$
by (\ref{formula: bound on c_ij}).

\begin{rmk}\label{rmk: lin alg}
Since $C, D$ are orthonormal, 
$d_1c_1+d_2c_2=0$ so 
\begin{eqnarray*}
(d_1c_2-d_2c_1)^2=(d_1c_2-d_2c_1)^2+(d_1c_1+d_2c_2)^2=
(d_1^2+d_2^2)(c_1^2+c_2^2)=c_1^2+c_2^2.
\end{eqnarray*}
In particular, if the mixed term vanishes and
the sectional curvatures of coordinate planes are bounded
above by a negative constant, then $K(\s)$ is bounded
above by the same constant as the coefficients add up to $1$:
\[
c_1^2+c_2^2+d_1^2c_3^2+d_1^2c_0^2+d_2^2c_3^2+d_2^2c_0^2
=c_1^2+c_2^2+(d_1^2+d_2^2)(c_0^2+c_3^2)
=1.
\]
\end{rmk}

\begin{rmk}\label{rmk: non-generic sigma}
If $dp(\s)$ is zero-dimensional, then $\s$  
is the $Y_1\ddr$-plane, whose sectional curvature
is given by (\ref{form: k(dr,y_i)}).
If $dp(\s)$ is one-dimensional, then $\s$ intersects the
$Y_1\ddr$-plane in a line, and we let $D$ be a unit
vector that spans the line, so $D=d_0\ddr+d_1Y_1$
with $d_0, d_1\in\mathbb R$. Let $C$ be a unit vector in $\s$
that is orthogonal to $D$. Then $dp(\s)$ is a nonzero subspace
spanned by $dp(C)$, and we let $H_2$ be a unit vector that 
is proportional to $dp(C)$. Extending $H_2$ to a frame 
$\{\check X_2,\dots, \check X_{2n-1}\}$ in $\chm$
satisfying $\check X_2=H_2$, we get a frame $\ddr, Y_1,\dots, Y_{2n-1}$
near $z$ in which $C=c_0\ddr+c_1Y_1+c_2Y_2$
with $c_0, c_1, c_2\in\mathbb R$. Repeating the above arguments, we
easily compute $K(C,D)$, and in fact all the mixed terms now
vanish so that
\begin{eqnarray}
\label{form: k(c,d) non-generic}
&\ \  \qquad K(C , D)=
(d_0c_1-d_1c_0)^2 K (\ddr,Y_1)+
d_0^2c_2^2 K(\ddr,Y_2)+d_1^2c_2^2 K(Y_1,Y_2),
\end{eqnarray}
where again $d_0c_0+d_1c_1=0$ implies that $(d_0c_1-d_1c_0)^2=c_1^2+c_2^2$,
so that if the sectional curvatures of coordinate planes are bounded
above by a negative constant, then $K(\s)$ is bounded
above by the same constant.
\end{rmk}

\section{Construction of the metric and curvature estimates}
\label{sec: curv estimate}

In this section we construct the functions $v, h$ such that
the metric $\l_{v,h}$ is complete, 
$\sec(\l_{v,h})$ is bounded above by a negative number, and
$\l_{v,h}$ agrees with the complex hyperbolic metric, i.e. 
$v=\sinh(r)$, $h=\cosh(r/2)$, when $r$ is at least half of
the normal injectivity radius of $S$. The domain of $v, h$
will be the interval from $-\infty$ to the normal 
injectivity radius of $S$.

Let $\e$ be a small positive parameter
such that $8\e$ is less than the normal
injectivity radius of $S$ in $M$. When precise estimates
are unimportant we use the ``big $O$'' notation, 
and rely on smallness of $\e$ without further mention.

{\bf Defining $v$ by bending $\sinh(r)$ to $\e e^r$.}   
Let $r_\e$ be the unique solution of the equation
$\sinh(r)=\e e^r$; thus $-2r_\e=\ln(1-2\e)$
so that $r_\e=\e+ O(\e^2)\approx \e$. Let $r_\e^-:=r_\e-\e^4$.

\begin{prop}
\label{prop: properties of bv}
There is a $C^1$ function $\bv$ and 
$r_\e^+\in (r_\e, r_\e+\e^4]$ such that\newline
\textup{(1)}
$\bv$ is positive and increasing,\newline
\textup{(2)}
$\bv(r)=\sinh(r)$ for $r\ge r_\e^+$, \newline
\textup{(3)}
$\bv(r)=\e e^r$ for $r\le r_\e^-$, \newline
\textup{(4)}
if $r\in [r_\e^-, r_\e^+]$, 
then $\bv$ is $C^\infty$, 
$\bv^{\prime\prime}(r)>\bv(r)$, and $(\ln(\bv))^{\prime\prime}>0$.

\end{prop}
\begin{proof}
The slope of $\ln(\sinh(r))$ at $r_\e$ is $\coth(r_\e)\gg 1$, 
so the graphs of $\ln(\sinh(r))$ 
and $r+\ln(\e)$ intersects transversely at $r_\e$.

Since $\ln(\sinh(r))^{\prime\prime}=-\frac{1}{\sinh^2(r)}<0$, 
the function $\ln(\sinh(r))$ is (strictly) concave, so 
given $r_\e^+\in (r_\e, r_\e+\e^4)$ the tangent line
$l_\e^+$ to $\ln(\sinh(r))$ at 
$r_\e^+$ intersects the line $l_\e^-(r)=r+\ln(\e)$ 
at $r_\e^0<r_\e$.
(This becomes obvious after drawing graphs of
$\ln(\sinh(r))$, $l_\e^-$ near $r_\e$. Alternatively,
the lines $l_\e^+$, $l_\e^-$ intersect transversely,
and they cannot intersect at a point $r\ge r_\e$ because 
$r\ge r_\e$ implies $l_\e^+(r)\ge \ln(\sinh(r))\ge l_\e^-(r)$, 
where the first inequality follows from concavity
of $\ln(\sinh(r))$, and the inequalities become
equalities at different points
$r_\e^+$, $r_\e$).

Since $r_\e^0\to r_\e$ as $r_\e^+\to r_\e$, we may assume that 
$r_\e^0\in (r_\e^-, r_\e)$. 
Note that $r_\e^0$ is the only nonsmooth point of the
piecewise-linear function $l(r):=\max\{l^-(r)$, $l^+(r)\}$.
The slope of $l^-$ is $1$, and the slope of $l^+$ is $\coth(r_\e^+)>1$,
so $l$ is convex. 
Restricting $l$ to $[r_\e^-, r_\e^+]$, we
let $w_l$ be the smoothing of $l$ given by 
Proposition~\ref{prop: bend} for some small $\delta$.
Thus $w_l$ is a $C^\infty$ increasing function defined on
$[r_\e^-, r_\e^+]$ and such that
$w_l^{\prime\prime}>0$, and the graphs of $l$, $w_l$ touch
at the points $r_\e^-$, $r_\e^+$. 

Let $w$ be the function
equal to $r+\ln(\e)$ for $r\le r_\e^-$, equal to $w_l$ for
$r\in [r_\e^-, r_\e^+]$, and equal to $\ln(\sinh(r))$ for $r\ge r_\e^+$.
Then $w$ is an increasing $C^1$ function, and the function 
$\bv:=e^w$ is positive, increasing, and $C^1$, and furthermore,
the restrictions of $\bv$ to $(-\infty, r_\e^-]$, $[r_\e^-, r_\e^+]$,
$[r_\e^+, \infty)$ are $C^\infty$.

Finally, assume $r\in [r_\e^-, r_\e^+]$, and consider the function $e^{w_l}$, 
i.e. the restriction of $\bv$ to $[r_\e^-, r_\e^+]$. 
Certainly, $(\ln(\bv))^{\prime\prime}=w_l^{\prime\prime}>0$.
Since $\frac{\bv^\prime}{\bv}=w_l^\prime$ is increasing, 
$0<(\frac{\bv^\prime}{\bv})^\prime=
\frac{\bv^{\prime\prime}}{\bv} - (\frac{\bv^\prime}{\bv})^2$.
Hence $\frac{\bv^{\prime\prime}}{\bv} >(\frac{\bv^\prime}{\bv})^2\ge 1$,
where the last inequality holds because $\frac{\bv^\prime}{\bv}$
is bounded below by its value at $r_\e^-$ which is equal to $1$,
because it can be computed using $\bv=\e e^r$.
\end{proof}

\begin{prop}\label{prop: defining v} 
For each small positive $\e$ there exists $\delta_0>0$, and a $C^\infty$ 
function $v=v(r)$ depending on the parameter $\delta\in (0,\delta_0)$ such that
\begin{itemize}
\item $v$ is positive and increasing,
\item $v(r)=\bv(r)$ if $r$ is outside the $\e^8$-neighborhood
of  $\{r_\e^-, r_\e^+\}$,  
\item 
if $r$ is in the $\e^8$-neighborhood of $\{r_\e^-, r_\e^+\}$,
then $\frac{v^{\prime\prime}}{v}> 1+O(\e)$,
\item 
if $\e$ is fixed, then
$v$ converges to $\bv$ in uniform $C^1$ topology as $\delta\to 0$.
\end{itemize}
\end{prop}
\begin{proof}
We define $v:=\bv_{\delta, \s}$ to be the smoothing 
of $\bv$ at $r_\e^-$, $r_\e^+$, given by Lemma~\ref{lem: smoothing conv}.
In particular, $v$ is positive and increasing,
$v=\bv$ outside the $\s$-neighborhood of 
$\{r_\e^-, r_\e^+\}$,
and $v$ converges to $\bv$ uniformly in $C^1$ topology as $\delta\to 0$.
If $r$ is in the $\s$-neighborhood of $[r_\e^-, r_\e^+]$, then 
$\bv(r)<\bv(r_\e^++2\s)$, so if 
$\delta$ is small enough, then
$v(r)<\bv(r_\e^++2\s)$.

By Proposition~\ref{prop: properties of bv}, 
if $r\in [r_\e^-, r_\e^+]$ 
then $\bv^{\prime\prime}(r)>\bv(r)> \bv(r_\e^--2\s)$,
where by $\bv^{\prime\prime}(r)$ at $r_\e^-, r_\e^+$,
we mean one-sided derivatives. 
If $r\in [r_\e^--\s, r_\e^++\s]\setminus (r_\e^-, r_\e^+)$,  
then $\bv$ equals to $\e e^r$ or $\sinh(r)$,
so $\bv^{\prime\prime}(r)=\bv(r)>\bv(r_\e^--2\s)$.
Therefore,
Lemma~\ref{lem: smoothing conv} implies that
$v^{\prime\prime}>\bv(r_\e^--2\s)$. 

Therefore, for small $\delta$ and $\s=\e^8$
\[
\frac{v^{\prime\prime}}{v}>
\frac{\bv(r_\e^--2\s)}{\bv(r_\e^-+2\s)}=
\frac{\e e^{r_\e^--2\s}}{\sinh(r_\e^-+2\s)}=1+O(\e).
 \]
provided $r$ lies in the $\e^8$-neighborhood of $\{r_\e^-, r_\e^+\}$.
\end{proof}

{\bf Defining $h$ by bending from $\cosh(r/2)$ to $e^{r/2}$.} 
Let $\r_\e=\frac{r_\e^-}{2}$ so that $\r_\e<r_\e^-=r_\e-\e^4\approx\e$, and
$\r_\e=\frac{\e}{2}+O(\e^2)\approx\frac{\e}{2}$.
The tangent line to the graph of $\cosh(\frac{r}{2})$ at $\r_\e$ is 
\begin{eqnarray}\label{form: l=tang line}
l(r)=\cosh\left(\frac{\r_\e}{2}\right)+
\frac{1}{2}\sinh\left(\frac{\r_\e}{2}\right)(r- \r_\e).
\end{eqnarray}
Let $q(r):=l(r)+\e^6(r- \r_\e)^2$, so that the graphs of $q$ and $l$ touch
at $\r_\e$.

\begin{prop}
\label{prop: properties of bh}
There is a $C^1$ function $\bh$ and $n_\e<m_\e<\r_\e$ such that\newline
\textup{(1)} $\bh$ is positive and increasing,\newline
\textup{(2)} $\bh(r)=\cosh(\frac{r}{2})$ for $r\ge \r_\e$, \newline
\textup{(3)} $\bh(r)=q(r)$ for $r\in [m_\e, \r_\e]$, \newline
\textup{(4)} if $r\in [n_\e, m_\e]$, then $\bh$ is $C^\infty$, 
$\bh^{\prime\prime}(r)>\bh(r)/4$, and 
$(\ln(\bh))^{\prime\prime}>0$, and
$\frac{\bh^\prime}{\bh}\in [\frac{1}{2}, \frac{3}{4}]$.\newline
\textup{(5)} if $r\le  n_\e$, then $\bh(r)=e^{r/2}$. 
\end{prop}
\begin{proof}
One verifies that $q$ is increasing on $[-\frac{1}{\e^2}, \r_\e]$, and
$q(-\frac{1}{{\e}^2})<0$ while $q(0)>0$
so the parabola $q$ has exactly one zero 
$z_\e$ in $(-\frac{1}{\e^2}, 0)$. Therefore,
on the interval $(z_\e, \r_\e]$ the slope $\frac{q^\prime}{q}$ of the
function $\ln(q)$ varies from $+\infty$ to 
$\frac{1}{2}\tanh(\frac{\r_\e}{2})=O(\e)$.
So $(z_\e, \r_\e]$ contains a point 
$m_\e$ where $\frac{q^\prime}{q}=\frac{3}{4}$.
Let $L^+$ be the tangent line to the graph of $\ln(q)$ at $m_\e$, and let 
$L^-(r)=r/2$. 

\begin{lem}\label{lem: q>exp(m/2)}
$L^+(m_\e)>L^-(m_\e)$.
\end{lem}
\begin{proof}[Proof of Lemma~\ref{lem: q>exp(m/2)}] 
Since $L^+(m_\e)=\ln(q(m_\e))$ and $L^-(r)=r/2$, 
we need to show that $q(m_\e)>e^{m_\e/2}$.
As $q\ge l$, it suffices to show that $l(m_\e)>e^{m_\e/2}$.
Using $q^\prime=3q/4$ at the point $m_\e$, and $q(r)=l(r)+\e^6(r-\r_\e)^2$
we derive that $l^\prime=3l/4+O(\e^2)$ at $m_\e$.
Denote $2l^\prime(m_\e)=\sinh(\frac{\r_\e}{2})$ by $x$;
note that $x=\e/4+O(\e^2)$. 
Then 
\[
l(m_\e)=\frac{4l^\prime(m_\e)}{3}+O(\e^2)=\frac{2x}{3}+O(\e^2),
\] 
while (\ref{form: l=tang line}) implies 
$l(m_\e)=1+m_\e\frac{x}{2}+O(\e^2)$.
Thus $m_\e=\frac{4}{3}-\frac{2}{x}+O(\e)<2-\frac{2}{x}$,
so $e^{m_\e/2}\le e^{1-\frac{1}{x}}$.
On the other hand, $l(m_\e)=\frac{2x}{3}+O(\e^2)>\frac{x}{2}$.
So it remains to show that $\frac{x}{2}>e^{1-\frac{1}{x}}$,
or equivalently $x e^\frac{1}{x}>2e$. 
One verifies that $x e^\frac{1}{x}$
decreases if $x\in (0, 1)$ and so for $0<x<\frac{1}{5}$, we get
$x e^\frac{1}{x}>\frac{e^5}{5}>2e$, and Lemma~\ref{lem: q>exp(m/2)} is proved.
\end{proof}

Combining Lemma~\ref{lem: q>exp(m/2)} with the fact that
the slope of  $L^+$ is $\frac{q^\prime}{q}(m_\e)=\frac{3}{4}$, and the slope
of $L^-$ is $\frac{1}{2}$, we get that $L^+(r)=L^-(r)$ 
for some $r<m_\e$, which is the 
only nonsmooth point of the convex
piecewise-linear function $L:=\max\{L^-, L^+\}$.
Let us fix an arbitrary $n_\e<r$.
Restricting $L$ to $[n_\e, m_\e]$, we
let $\omega_L$ be the smoothing of $L$ given by 
Proposition~\ref{prop: bend} for some small $\delta$.
Thus $\omega_L$ is a $C^\infty$ increasing function defined on
$[n_\e, m_\e]$ and such that
$\omega_L^{\prime\prime}>0$, and the graphs of $L$, $\omega_L$ touch
at the points $n_\e$, $m_\e$. 

Let $\omega$ be the function
equal to $\frac{r}{2}$ for $r\le n_\e$, equal to $\omega_L$ for
$r\in [n_\e, m_\e]$, equal to $\ln(q)$ for $r\in [m_e, \r_\e]$,
and equal to $\ln(\cosh(\frac{r}{2}))$ for $r\ge \r_\e$.
Then $\omega$ is an increasing $C^1$ function, 
and the function $\bh=e^\omega$ is positive, increasing, $C^1$,
and furthermore the restrictions of $\bh$ to $(-\infty, n_\e]$,
$[n_\e, m_\e]$, $[m_\e, \r_\e]$, $[\r_\e, +\infty)$ are $C^\infty$,
and furthermore, $\bh (r)=e^{r/2}$ if $r\le n_\e$, and $\bh=q$ if
$r\in [m_\e, \r_\e]$, and $\bh=\cosh(\frac{r}{2})$ if $r\ge \r_\e$. 

Finally, assume $r\in [n_\e, m_\e]$, and consider the function $e^{\omega_L}$, 
i.e. the restriction of $\bh$ to $[r_\e^-, r_\e^+]$. 
Certainly, $(\ln(\bh))^{\prime\prime}=w_L^{\prime\prime}>0$.
Since $\frac{\bh^\prime}{\bh}=w_L^\prime$ is increasing, 
$0<(\frac{\bh^\prime}{\bh})^\prime=
\frac{\bh^{\prime\prime}}{\bh} - (\frac{\bh^\prime}{\bh})^2$.
Hence $\frac{\bh^{\prime\prime}}{\bh} >
(\frac{\bh^\prime}{\bh})^2\ge\frac{1}{4}$,
where the last inequality holds because $\frac{\bh^\prime}{\bh}$
is bounded below by its value at $n_\e$ which equals to $\frac{1}{2}$,
because it can be computed using $\bh(r)=e^{r/2}$.
Since $\frac{\bh^\prime}{\bh}=\omega_L^\prime$ 
is increasing, $\frac{\bh^\prime}{\bh}$ varies on $[n_\e, m_\e]$
between its values at endpoints, 
where $\bh$ equals to $e^{r/2}$ and $q$, respectively,
hence $\frac{\bh^\prime}{\bh}\in [\frac{1}{2}, \frac{3}{4}]$ 
when $r\in [n_\e, m_\e]$.
\end{proof}

\begin{prop} 
\label{prop: defining h} For each small $\e$
and each $\s\in (0,\e^8)$ there is
$\delta_0>0$, and there exists a $C^\infty$ function $h=h(r)$ 
depending on the parameters
$\e$, $\s$, and $\delta\in (0,\delta_0)$
such that 
\begin{itemize}
\item $h$ is positive and increasing,
\item $h(r)=\bh(r)$ if $r$ is outside the $\s$-neighborhood
of  $\{n_\e, m_\e, \r_\e\}$,  
\item 
if $r$ is in the $\s$-neighborhood of $\{m_\e, \r_\e\}$,
then $\frac{h^{\prime\prime}}{h}>\e^6$.
\item 
if $r$ is in the $\s$-neighborhood of $n_\e$,
then $\frac{h^{\prime\prime}}{h}> \frac{1}{9}$,
\item  
if $\e, \s$ are fixed, then
$h$ converges to $\bh$ in uniform $C^1$ topology as $\delta\to 0$.
\end{itemize}
\end{prop}
\begin{proof}
Let $h:=\bh_{\delta, \s}$ be the smoothing 
of $\bh$ at $n_\e$, $m_\e$, $\r_\e$, given by Lemma~\ref{lem: smoothing conv}.
In particular, $h$ is positive and increasing,
$h=\bh$ outside the $\s$-neighborhood of $\{n_\e, m_\e, \r_\e\}$,
and $h$ converges to $\bh$ uniformly in $C^1$ topology as $\delta\to 0$.

To establish the desired lower bounds on $\frac{h^{\prime\prime}}{h}$
we need to look at one-sided second derivatives
$\bh^{\prime\prime}$ and then apply Lemma~\ref{lem: smoothing conv}
to derive a lower bound on $h^{\prime\prime}$. 
In the $\s$-neighborhood of $\r_\e$ 
the one-sided second derivatives satisfy 
\[
\bh^{\prime\prime}\ge\min\{2\e^6,\frac{1}{4}
\cosh\left(\frac{\r_\e}{2}\right)\}=2\e^6,
\] 
so  Lemma~\ref{lem: smoothing conv}
implies that $h^{\prime\prime}>\frac{3\e^6}{2}$ for small $\delta$. 
As $h(r)<\bh(\r_\e+2\s)$ for small $\delta$, we conclude that
$\frac{h^{\prime\prime}}{h}> 
\frac{3\e^6}{2\cosh\left(\frac{\r_\e+2\s}{2}\right)}>\e^6$
where the last inequality holds if, say, $\s<\e^8$.

By Proposition~\ref{prop: properties of bh},
if $r\in [n_\e, m_\e]$, then $\bh^{\prime\prime}(r)>\frac{\bh(r)}{4}$.
So if $r\in [n_\e, n_\e+\s]$, then 
$\bh^{\prime\prime}(r)>\frac{\bh(r)}{4}>\frac{\bh(n_\e-2\s)}{4}$,
while if $r\in [n_\e-\s, n_\e]$, then 
$\bh^{\prime\prime}=\frac{\bh}{4}>\frac{\bh(n_\e-2\s)}{4}$.
So if $r$ is in the $\s$-neighborhood of $n_\e$, and $\delta$ 
is small, then Lemma~\ref{lem: smoothing conv} implies 
$h^{\prime\prime}(r)>\frac{\bh(n_\e-2\s)}{4}$ and 
$h(r) <\bh(n_\e+2\s)$, and thus
\[
\frac{h^{\prime\prime}}{h}> \frac{\bh(n_\e-2\s)}{4\bh(n_\e+2\s)}>\frac{1}{9},
\]
where the last inequality holds provided $\s$ is made small
while $\e$ is kept fixed.

Similarly, if $r\in [m_\e-\s, m_\e]$, then
$\bh^{\prime\prime}(r)>\frac{\bh(r)}{4}>\frac{\bh(m_\e-2\s)}{4}$,
while if $r\in [m_\e, m_\e+\s]$, then $\bh^{\prime\prime}(r)=2\e^6$.
As $\s\to 0$ we have
\[
\bh(m_\e-\s)\to \bh(m_\e)=q(m_\e)>l(m_\e)=\frac{\e}{6}+O(\e^2).
\]
So for small $\s$ we get $\bh^{\prime\prime}(r)> 2\e^6$ and 
$\bh(r)\le q(m_\e+\s)<2$ on the $\s$-neighborhood of $m_\e$.
 Thus if $\delta$ is small, 
$\frac{h^{\prime\prime}}{h}>\e^6$ on the $\s$-neighborhood of $m_\e$.
\end{proof}

\begin{thm}\label{thm: main curv estimates}
For any sufficiently small positive $\e$ there are small positive 
$\s$, $\delta$, and a negative constant $M_{\e,\s, \delta}$
such that $K(\l_{v,h})\le M_{\e,\s, \delta}$.
\end{thm}

\begin{rmk} More precisely, there are ranges of
$\e$, $\s$, $\delta$ for which Theorem~\ref{thm: main curv estimates}
holds, namely, $\e\in (0,\e_0)$, $\s\in (0, \s_0(\e))$, and
$\delta\in (0, \delta_0(\e,\s))$, i.e. the range of $\s$
depends on $\e$ and the range on $\delta$ depends on $\e, \s$.
\end{rmk}

\begin{proof} 
It is enough to give a proof for $2$-planes that project isomorphically 
to $\chm$, because they form a dense subset in every tangent space.
The points $r_\e^+$, $r_\e^-$, $\r_\e$, $m_\e$, $n_\e$
divide the real line into six intervals, and we estimate the curvature
on each interval separately.

{\bf Step 0.} Suppose $r\ge r_\e^+$.
Then $\bv=\sinh(r)$ and $h=\bh=\cosh(\frac{r}{2})$, 
and $v$ converges to $\bv$ in $C^1$ topology
as $\delta\to 0$, and $\frac{v^{\prime\prime}}{v}>1+O(\e)$. 
If $v$ were equal to $\bv$, then the metric would be complex hyperbolic
giving $K(C, D)\le -\frac{1}{4}$. In general,
the formulas (\ref{form: k(c,d)})--(\ref{form: r(dr,y_1, y_2, y_3)})
immediately imply that the upper curvature bound for
$K(C, D)$ converges to $-\frac{1}{4}+O(\e)$, as $\delta\to 0$, so that
$K(C, D)\le -\frac{1}{5}$ for all sufficiently small $\e$, $\delta$.

{\bf Step 1.} Suppose $r\in [r_\e^-, r_\e^+]$.
Then $h=\bh=\cosh(\frac{r}{2})$, and $\bv$ is positive, increasing
and $(\ln(\bv))^{\prime\prime}>0$. Furthermore,
$v$ converges to $\bv$ in $C^1$ topology 
as $\delta\to 0$, and $\frac{v^{\prime\prime}}{v}>1+O(\e)$.  

Since $\bv$ is increasing, $\bv(r)\le v(r_\e+\e^4)=\sinh(r_\e+\e^4)$,
and since $\bh$ is increasing, we have $h(r)\ge\cosh(\frac{r_\e-\e^4}{2})$
so 
\[
\frac{\bv}{h^2}\le\frac{\sinh(r_\e+\e^4)}{\cosh(\frac{r_\e-\e^4}{2})}=
\e+O(\e^2)<2\e.
\]
Also since $\frac{\bv^\prime}{\bv}$ is increasing, it can be estimated 
at endpoints where $\bv$ is equal to $\e e^r$, $\sinh(r)$. Thus
\[
1\le \frac{\bv^\prime}{\bv}\le\coth(r_\e+\e^4).
\]
On the other hand, 
$\frac{h^\prime}{h}=\frac{1}{2}\tanh(\frac{r}{2})=\frac{\e}{4}+O(\e^2)$
is small and positive, in particular,
\[
0<\frac{\bv^\prime}{\bv}-\frac{h^\prime}{h}\le\coth(r_\e+\e^4)\ \ \
\text{and}\ \ \
\frac{h^\prime}{h}\frac{\bv^\prime}{\bv}\ge
\frac{\e}{4}+O(\e^2)>\frac{\e}{5}. 
\]
Also 
\[
\frac{\bv}{h^2}\left(\frac{\bv^\prime}{\bv}-\frac{h^\prime}{h}\right)\le
\frac{\sinh(r_\e+\e^4)}{\cosh(\frac{r_\e-\e^4}{2})}
\coth(r_\e+\e^4)=1+O(\e^2).
\]

%
If $v$ were equal to $\bv$, as happens for $r\in [r_\e^-+\e^8, r_\e^+-\e^8]$,
then the above estimates would imply the following
\begin{eqnarray}
& \label{form: step1-mixed}
|\langle R(\ddr , Y_1) Y_2, Y_3\rangle |\le |c_{23}|+O(\e^2)<|c_{23}|+\e,\\
& \label{form: step1-k(yiy1)}
K (Y_2,Y_1)=K(Y_3,Y_1)=
\frac{\bv^2}{16h^4}-\frac{\bv^\prime}{\bv}\frac{h^\prime}{h}<
-\frac{\e}{5}<0,\\
& \qquad\ K(Y_3, Y_2)<
-\frac{1}{h^2}\left(\frac{1}{4}+3c_{23}^2\right)<
-\frac{1}{\cosh^2(r_\e^+)}\left(\frac{1}{4}+3c_{23}^2\right)<
-\left(\frac{1}{5}+3c_{23}^2\right),
\end{eqnarray}
and since the inequalities are strict, they hold for $v$ in place of $\bv$ 
provided $\delta$ is made small while $\e$ is kept fixed,
so from now on we switch to $v$. 
Since $\frac{v^{\prime\prime}}{v}>1+O(\e)$ and
$h=\cosh(\frac{r}{2})$, we get  
%
\begin{eqnarray*}
K(\ddr, Y_1)<-1-O(\e)
\qquad\text{and}\qquad
K(\ddr, Y_2)=-\frac{1}{4}.
\end{eqnarray*}
From the formula (\ref{form: k(c,d)}) we get
\begin{eqnarray}
\label{form: step1} 
& K(C , D)\le m(C, D, \e):= 
-\frac{\e}{5}\left((d_1c_2-d_2c_1)^2 + d_1^2c_3^2\right)- \\
& \nonumber
(1+O(\e))d_1^2c_0^2
-d_2^2c_3^2\left(\frac{1}{5}+3c_{23}^2\right)
-\frac{1}{4}d_2^2c_0^2 
+3\left(|c_{23}|+\e\right)|d_1d_2c_0c_3|=\\
& \nonumber -\frac{\e}{5}\left((d_1c_2-d_2c_1)^2 + d_1^2c_3^2\right) -
\left(|d_1c_0|\sqrt{1+O(\e)}-|d_2c_3|\sqrt{\frac{1}{5}+3c_{23}^2}\right)^2-\\
&  \nonumber
\frac{1}{4}d_2^2c_0^2+
|d_1d_2c_0c_3|\left(3|c_{23}|+3\e -2\sqrt{1+O(\e)}
\sqrt{\frac{1}{5}+3c_{23}^2}\right),
\end{eqnarray}
where $3|c_{23}|<2\sqrt{1+O(\e)}\sqrt{\frac{1}{5}+3c_{23}^2}$ so every
summand in $m(C, D, \e)$ is nonpositive for small $\e$. 
In fact, if $\e$ is sufficiently small and positive, 
then $m(C, D, \e)<0$. (Otherwise, every summand 
would have to vanish. In particular, 
$|d_1c_0|=|d_2c_3|\sqrt{\frac{1}{5}+3c_{23}^2}$, and 
$d_2c_0=0$. 
The equation $d_2c_0=0$ would imply that 
either $d_2$ or $c_0$ vanishes.
If $d_2=0$, then $|d_1|=1$ and $c_0=0$, so that
$(d_1c_2-d_2c_1)^2 + d_1^2c_3^2=c_1^2+c_2^2+c_3^2=1$. 
If $d_2\neq 0$, then $c_0=0$, and hence $c_3=0$, so 
$(d_1c_2-d_2c_1)^2 + d_1^2c_3^2=c_1^2+c_2^2=1$.
So in either case we get a contradiction with the fact that
every summand vanishes).

Let $M_1(\e)$ be the maximum of $m(C, D, \e)$ over all orthonormal 
$C, D$; by compactness the maximum is attained, i.e. 
$M_1(\e)=m(C^*, D^*, \e)$ for some $C^*$, $D^*$, and 
by the previous paragraph, $m(C^*, D^*, \e)<0$, so 
$K(C,D)\le M_1(\e)<0$ for all $C, D$ and all small positive $\e$.

{\bf Step 2.}
Suppose $r\in [\r_\e, r_\e^-]$. If $r$ is not in the $\s$-neighborhood of
$\{\r_\e, r_\e^-\}$, then $v(r)=\e e^r$
and $h(r)=\cosh(\frac{r}{2})$, and in general 
$v, h$ converge to $\e e^r$, $\cosh(\frac{r}{2})$ in $C^1$-topology
as $\delta\to 0$, and furthermore 
by Propositions~\ref{prop: defining v}, \ref{prop: defining h}
$\frac{v^{\prime\prime}}{v}\ge 1+O(\e)>\frac{1}{4}$
and $\frac{h^{\prime\prime}}{h}>\e^6$.
Then one verifies that $\frac{v}{h^2}<2\e$ and 
$\frac{h^\prime}{h}> \frac{\e}{9}$ for small $\e, \delta$.
The formulas 
(\ref{form: k(y_i,y_1)})--(\ref{form: r(dr,y_1, y_2, y_3)})
give the following.  
\begin{eqnarray*} 
& 
K (Y_2,Y_1)=K(Y_3,Y_1)<-\frac{\e}{10},\\
& 
K (Y_3,Y_2)<
-\frac{1}{4h^2}< -\frac{1}{9},\\
& K(\ddr,Y_1)<-\frac{1}{4}, \ \ \ \ \
K(\ddr,Y_2)<-\e^6,\\
& 
|\langle R(\ddr , Y_1) Y_2, Y_3\rangle|=
|c_{23}|(2\e+O(\e^2))\le\e+O(\e^2)<2\e.
\end{eqnarray*}
Thus $K(C , D)$ is bounded above by
\begin{eqnarray*}
&  -\frac{\e}{10}\left((d_1c_2-d_2c_1)^2 + d_1^2c_3^2\right)
-\frac{1}{4}d_1^2c_0^2-\e^6 d_2^2c_0^2
-\frac{1}{9}d_2^2c_3^2 +6\e |d_1d_2c_0c_3|=\\
& \nonumber -\frac{\e}{10}\left((d_1c_2-d_2c_1)^2 + d_1^2c_3^2\right)
-\e^6d_2^2c_0^2-(\frac{1}{2}|d_1c_0|-\frac{1}{3}|d_2c_3|)^2+
|d_1d_2c_0c_3|(6\e-\frac{1}{3}),
\end{eqnarray*}
in which every summand is nonpositive. Then the argument
of Step $1$ gives a function $M_2(\e)$ such that
$K(C,D)\le M_2(\e)<0$ for all $C, D$ and all small positive $\e$. 

{\bf Step 3.} Suppose $r\in [m_\e, \r_\e]$ so that $v(r)=\bv=\e e^r$,
$\bh=q$, and $\frac{h^{\prime\prime}}{h}>\e^6$.
If $r$ is outside the $\s$-neighborhood of $\{m_\e, \r_\e\}$ then
$h=q$, and on the whole interval $ h$ converges to $q$ in $C^1$-topology
as $\delta\to 0$. 

On the interval $[m_\e, \r_\e]$
one computes that 
$q^\prime=\frac{\e}{8}+O(\e^2)$ and hence $q^\prime>0$,
so that $q(r)< q(\r_\e)=\cosh(\frac{\r_\e}{2})=1+O(\e^2)$,
while $q^{\prime\prime}=2\e^6$, and therefore 
$\left(\frac{q^\prime}{q}\right)^\prime=
\frac{qq^{\prime\prime}-(q^\prime)^2}{q^2}<0$, i.e. 
$\frac{q^\prime}{q}$ decreases on $[m_\e, \r_\e]$ from $\frac{3}{4}$
to $\frac{1}{2}\tanh(\frac{\r_\e}{2})=\frac{\e}{8}+O(\e^3)$, the values of
$\frac{q^\prime}{q}$ at the endpoints of $[m_\e, \r_\e]$.
Thus if $\delta$ is small, then
$\frac{h^\prime}{h}\in (\frac{\e}{9}, \frac{4}{5})$ 
on $[m_\e, \r_\e]$.

Furthermore,
one computes that $\frac{v}{\bh^2}=\frac{\e e^r}{q^2}$ satisfies
\[
\left(\frac{\e e^r}{q^2}\right)^\prime=2\frac{\e e^r}{q^2}
\left(\frac{1}{2}-\frac{q^\prime}{q}\right)\ \ \ \text{and}\ \ \
\left(\frac{\e e^r}{q^2}\right)^{\prime\prime}=2\frac{\e e^r}{q^2}\left(
2\left(\frac{1}{2}-\frac{q^\prime}{q}\right)^2-
\left(\frac{q^\prime}{q}\right)^\prime\right)>0.
\]
So the point where $\frac{q^\prime}{q}=\frac{1}{2}$
is the global minimum of $\frac{e^r}{q^2}$, and the maximum 
is attained at the endpoints. We conclude that
\[
\frac{v}{\bh^2}\le
\e\cdot\max
\left\{
\frac{e^{m_e}}{q(m_\e)^2}, 
\frac{e^{\r_\e}}{\cosh^2(\frac{\r_\e}{2})}\right\}=
\e \frac{e^{\r_\e}}{\cosh^2(\frac{\r_\e}{2})}<2\e.
\]
where the equality in the middle holds because
$e^{m_e}<q(m_\e)^2$ by Lemma~\ref{lem: q>exp(m/2)}. 
Hence for small $\delta$ we have $\frac{v}{h^2}<2\e$.
In summary, for small $\e,\delta$
the above estimates combined with formulas 
(\ref{form: k(y_i,y_1)})--(\ref{form: r(dr,y_1, y_2, y_3)})
imply the following.
\begin{eqnarray*} 
& 
K (Y_2,Y_1)=K(Y_3,Y_1)\le
\frac{\e^2}{4}-\frac{\e}{9}<-\frac{\e}{10},\\
& 
K (Y_3,Y_2)<-\frac{1}{4h^2}<
-\frac{1}{4\cosh^2(\r_\e)}
<-\frac{1}{9},\\
& 
K(\ddr,Y_1)=-1, \ \ \ \ \
K(\ddr,Y_2)\le -\e^6,\\
& 
|\langle R(\ddr , Y_1) Y_2, Y_3\rangle|\le 
|c_{23}|2\e(1-\frac{\e}{9})<\e.
\end{eqnarray*}

From the formula (\ref{form: k(c,d)}) we conclude that
$K(C , D)$ is bounded above by
\begin{eqnarray*}
&  -\frac{\e}{10}\left((d_1c_2-d_2c_1)^2 + d_1^2c_3^2\right)
-d_1^2c_0^2-\e^6 d_2^2c_0^2
-\frac{1}{9}d_2^2c_3^2 +3\e |d_1d_2c_0c_3|=\\
& \nonumber -\frac{\e}{10}\left((d_1c_2-d_2c_1)^2 + d_1^2c_3^2\right)
-\e^6d_2^2c_0^2-(|d_1c_0|-\frac{1}{3}|d_2c_3|)^2+
|d_1d_2c_0c_3|(3\e-\frac{1}{3}),
\end{eqnarray*}
in which every summand is nonpositive. Then the argument as in 
Step $1$ gives a function $M_3(\e)$ such that
$K(C,D)\le M_3(\e)<0$ for all $C, D$ and all small positive $\e$.

{\bf Step 4.}
Suppose $r\in [n_\e, m_\e]$ so that $v(r)=\e e^r$, the function
$\frac{\bh^\prime}{\bh}$ 
is increasing, $h$ converges to $\bh$ in $C^1$-topology
as $\delta\to 0$, and furthermore,
$\frac{h^{\prime\prime}}{h}>\e^6$.

The values of $\frac{\bh^\prime}{\bh}$ at endpoints $n_\e, m_\e$
are $\frac{1}{2}$, $\frac{3}{4}$, respectively. 
So $\frac{v\prime}{v}-\frac{\bh^\prime}{\bh}\le\frac{1}{2}$.
Hence $\frac{v\prime}{v}-\frac{h^\prime}{h}<1$ and $\frac{h^\prime}{h}>\frac{1}{3}$
for small $\delta$.

Since $\ln(\bh)$ is convex, the graph of $\ln(\bh)$
is above its tangent line at $n_\e$, i.e. $\ln(\bh(r))\ge r/2$,
so that $\bh(r)\ge e^{r/2}$. 
It follows that
$\frac{v}{\bh^2}\le \frac{\e e^r}{e^r}=\e$
so that $\frac{v}{h^2}<2\e$ for small $\delta$.
The above estimates combined with formulas 
(\ref{form: k(y_i,y_1)})--(\ref{form: r(dr,y_1, y_2, y_3)})
imply the following.
\begin{eqnarray*} 
& 
K (Y_2,Y_1)=K(Y_3,Y_1)<
\frac{\e^2}{4}-\frac{1}{3}<-\frac{1}{4},\\
& 
K (Y_3,Y_2)<-\left(\frac{h^\prime}{h}\right)^2\le -\frac{1}{9},\\
& 
K(\ddr,Y_1)=-1, \ \ \ \ \
K(\ddr,Y_2)<-\e^6,\\
& 
|\langle R(\ddr , Y_1) Y_2, Y_3\rangle|\le 
|c_{23}|2\e\le \e.
\end{eqnarray*}
From the formula (\ref{form: k(c,d)}) we conclude that
$K(C , D)$ is bounded above by
\begin{eqnarray*}
&  -\frac{1}{4}\left((d_1c_2-d_2c_1)^2 + d_1^2c_3^2\right)
-d_1^2c_0^2-\e^6 d_2^2c_0^2
-\frac{1}{9}d_2^2c_3^2 +3\e |d_1d_2c_0c_3|=\\
& \nonumber -\frac{1}{4}\left((d_1c_2-d_2c_1)^2 + d_1^2c_3^2\right)
-\e^6d_2^2c_0^2-(|d_1c_0|-\frac{1}{3}|d_2c_3|)^2+
|d_1d_2c_0c_3|(3\e-\frac{1}{3}),
\end{eqnarray*}
in which every summand is nonpositive. Then the argument as in 
Step $1$ gives a function $M_4(\e)$ such that
$K(C,D)\le M_4(\e)<0$ for all $C, D$ and all small positive $\e$.

{\bf Step 5.} Suppose $r\le n_\e$ so that $v(r)=\e e^r$,
$\bh(r)=e^{r/2}$, the function $h$ converges to $\bh$ in $C^1$-topology
as $\delta\to 0$, and furthermore,
$\frac{h^{\prime\prime}}{h}>\frac{1}{9}$.
Hence $\frac{\bh^\prime}{\bh}=\frac{1}{2}$ and 
$\frac{v}{\bh^2}=\e$ implying 
$\frac{h^\prime}{h}>\frac{1}{3}$ and 
$\frac{v}{h^2}<2\e$.

Plugging into formulas 
(\ref{form: k(y_i,y_1)})--(\ref{form: r(dr,y_1, y_2, y_3)})
we get the following.  
\begin{eqnarray*} 
& 
K (Y_2,Y_1)=K(Y_3,Y_1)<
\frac{\e^2}{4}-\frac{1}{3}<-\frac{1}{4},\\
& 
K (Y_3,Y_2)<-\left(\frac{h^\prime}{h}\right)^2\le -\frac{1}{9},\\
& 
K(\ddr,Y_1)=-1, \ \ \ \ \
K(\ddr,Y_2)<-\frac{1}{9},\\
& 
|\langle R(\ddr , Y_1) Y_2, Y_3\rangle|\le 
|c_{23}|2\e(1-\frac{1}{3})< \e.
\end{eqnarray*}
From the formula (\ref{form: k(c,d)}) we conclude that
$K(C , D)$ is bounded above by
\begin{eqnarray*}
&  -\frac{1}{4}\left((d_1c_2-d_2c_1)^2 + d_1^2c_3^2\right)
-d_1^2c_0^2-\frac{1}{9}d_2^2c_0^2
-\frac{1}{9}d_2^2c_3^2 +3\e |d_1d_2c_0c_3|
%
\end{eqnarray*}
which is bounded above by $-\frac{1}{9}+3\e<-\frac{1}{10}$ because 
$|d_1d_2c_0c_3|\le 1$ and
\begin{eqnarray}
(d_1c_2-d_2c_1)^2+d_1^2c_3^2+d_1^2c_0^2+d_2^2c_3^2+d_2^2c_0^2=1,
\end{eqnarray}
which completes the proof.
\end{proof}

\begin{rmk} By a standard argument, 
recorded in Section~\ref{sec: proof of main thm}, 
the metric $\l_{v,h}$ constructed in 
Theorem~\ref{thm: main curv estimates} is complete.
\end{rmk}

\section{$A$-regular metrics of negative curvature}
\label{sec: A-reg}

The metric $\l_{v,h}$ constructed in 
Theorem~\ref{thm: main curv estimates} is not $A$-regular
because by definition any $A$-regular metric has a two-sided sectional
curvature bound, so if $\l_{v,h}$ were $A$-regular, it would
be negatively pinched, which is ruled
out by part (13) of Corollary~\ref{thm: intro-gr-theoretic-cor}.
In this section we modify $\l_{v,h}$ outside a large
compact set so that the new metric is $A$-regular
and has negative sectional curvature, which as we just explained 
cannot be bounded away from zero.

Let $\tau_\e:=\e e^{n_\e}$; note that $0<\tau_\e< 2\e$ because
$n_\e<\r_\e<\e$. Therefore, 
the parameters $o_\e:=\ln(\tau_\e)$ and $p_\e:=2\ln(\tau_\e)$
are negative, and go to $-\infty$ as $\e\to 0$, and moreover, 
$p_\e<o_\e=\ln(\e)+n_\e\ll n_\e$.
Let $F(r):=\frac{1}{2}\frac{e^{r/2}}{\tau_\e + e^{r/2}}$; 
this is the derivative of $\ln(\tau_\e+e^{r/2})$. 
Note that $F^\prime>0$, $F\in (0,\frac{1}{2})$, and
$F(p_\e)=\frac{1}{4}$.

\begin{prop} \label{prop: properties of bg}
For each small $\e>0$ there is a $C^1$ function
$\bg$ such that 
\begin{itemize}
\item
$\bg$ is positive and increasing,
\item 
if $r\ge o_\e$, then $\bg$ coincides with the function $h$
of Proposition~\ref{prop: defining h}, and in particular,
$\bg(r)=e^{r/2}$ for $r\in [o_\e, o_\e+1]$,
\item
$\bg(r)=\tau_\e+e^{r/2}$ for $r\in (-\infty, p_\e]$, 
\item
if $r\in [p_\e, o_\e]$, 
then $\bg$ is $C^\infty$, and $\frac{\bg^\prime}{\bg}$ is increasing, 
and $\frac{\bg^\prime}{\bg}\in [\frac{1}{4},\frac{1}{2}]$, and 
$\frac{\bg^{\prime\prime}}{\bg}>
\left(\frac{\bg^{\prime}}{\bg}\right)^2\ge \frac{1}{16}$,
\end{itemize}
\end{prop}
\begin{proof} 
The function $\bg$ is defined outside of $(p_\e, o_\e)$ 
so we just need to interpolate on this interval. Since
$\ln(\bg)$ equals to $r/2$ on $[o_\e, o_\e+1]$,
it coincides with its tangent line $l^+(r)=r/2$
at $o_\e$.
Let $l^-(r)=\ln(2\tau_\e)+\frac{1}{4}(r-2\ln(\tau_\e))$,
i.e. $l^-$ is the tangent line to the graph of
$\ln(\tau_\e+e^{r/2})$ at the point $p_\e=2\ln(\tau_\e)$. Then
\[
l^-(p_\e)=l^-(2\ln(\tau_\e))=\ln(2\tau_\e)>\ln(\tau_\e)=l^+(p_\e).
\]
On the other hand,
\[
l^-(o_\e)=l^-(\ln(\tau_\e))=\ln 2+\frac{3}{4}\ln(\tau_\e)<
\frac{1}{2}\ln(\tau_\e)=l^+(o_e),
\]
hence the lines $l^-, l^+$
intersect on the interval $(p_\e, o_\e)$.
The slope of $l^-$ is $\frac{1}{4}$ which is smaller
that the slope of $l^+$, 
thus the function $l:=\max\{l^-, l^+\}$ is convex and increasing.
Restricting $l$ to $[p_\e, o_\e]$, we
let $w_l$ be the smoothing of $l$ given by 
Proposition~\ref{prop: bend} for some small $\delta$.
Thus $w_l$ is a $C^\infty$ increasing function defined on
$[p_\e, o_\e]$
and such that
$w_l^{\prime\prime}>0$, and the graphs of $l$, $w_l$ touch
at the points $p_\e$, $o_\e$.

Let $w$ be the function
equal to $\ln(\tau_\e+e^{r/2})$ for $r\le p_\e$, 
equal to $w_l$ for $r\in [p_\e, o_\e]$, 
and equal to $\ln(h)$ for $r\ge o_\e$, where $h$ is the function 
of Proposition~\ref{prop: defining h}.
Then $w$ is an increasing $C^1$ function, and the function 
$\bg:=e^w$ is positive, increasing, $C^1$, and furthermore,
the restrictions of $\bg$ to $(-\infty, p_\e]$, 
$[p_\e, o_\e]$, $[o_\e, \infty)$ are $C^\infty$.

Finally, assume $r\in [p_\e, o_\e]$, and consider the function $e^{w_l}$, 
i.e. the restriction of $\bg$ to  $[p_\e, o_\e]$. 
Certainly, $(\ln(\bg))^{\prime\prime}=w_l^{\prime\prime}>0$, in other words,
$\frac{\bg^\prime}{\bg}=w_l^\prime$ is increasing, hence
it can be estimated at the endpoints $p_\e$, $o_\e$ where $\bg$ equals
to $\tau_\e+e^{r/2}$, $e^{r/2}$ so that the slopes of $\frac{\bg^\prime}{\bg}$
at $p_\e$, $o_\e$ are $\frac{1}{4}$, $\frac{1}{2}$, respectively.
Also $0<(\frac{\bg^\prime}{\bg})^\prime=
\frac{\bg^{\prime\prime}}{\bg} - (\frac{\bg^\prime}{\bg})^2$.
Hence $\frac{\bg^{\prime\prime}}{\bg} >(\frac{\bg^\prime}{\bg})^2\ge 
F(p_\e)^2=\frac{1}{16}$.
\end{proof}

\begin{prop} For each small $\e>0$ and
each $\s\in (0,\e^8)$ there is $\delta_0>0$, and
a $C^\infty$ function $g=g(r)$ depending on parameters 
$\e$, $\s$, and $\delta\in (0,\delta_0)$ such that 
\begin{itemize}
\item $g$ is positive and increasing,
\item $g(r)=\bg(r)$ if $r$ is outside the $\s$-neighborhood
of  $\{p_\e, o_\e\}$,
\item
if $r$ is in the $\s$-neighborhood of $[p_\e, o_\e]$,
then $\frac{g^{\prime\prime}}{g}> \frac{1}{25}$,
\item 
if $\e$, $\s$ are fixed, then
$g$ converges to $\bg$ in uniform $C^1$ topology as $\delta\to 0$.
\end{itemize}
\end{prop}
\begin{proof}
We let $g:=\bg_{\delta, \s}$ be the smoothing 
of $\bg$ at $p_\e$, $o_\e$, given by Lemma~\ref{lem: smoothing conv}.
In particular, $g$ is positive and increasing,
$g=\bg$ is outside the $\s$-neighborhood of $\{p_\e, o_\e\}$,
and $g$ converges to $\bg$ uniformly in $C^1$ topology as $\delta\to 0$.
Suppose $r\in [p_\e-\s, p_\e]$. 
Since $\left(\frac{\bg^{\prime}}{\bg}\right)^\prime=F^\prime>0$,  we get
\[
\frac{\bg^{\prime\prime}}{\bg}>
\left(\frac{\bg^{\prime}}{\bg}\right)^2= F^2>F^2(p_\e-2\s)>\frac{1}{25}
\]
for small $\s$. 
By Proposition~\ref{prop: properties of bg} the same lower bound
holds on $[p_\e, o_\e]$, 
i.e. $\frac{\bg^{\prime\prime}}{\bg}>F^2(p_\e-2\s)>\frac{1}{25}$
for small $\s$. 
Finally, if $r\in [o_n, o_n+\s]$, then 
$\frac{\bg^{\prime\prime}}{\bg}=\frac{1}{4}> F^2(p_\e-2\s)$ for small $\s$.
Thus by Lemma~\ref{lem: smoothing conv} we have 
$\frac{g^{\prime\prime}}{g}>\frac{1}{25}$ all small $\s, \delta$,
and $r$ in the $\s$-neighborhood of $[p_\e, o_\e]$.
\end{proof}

\begin{thm} \label{thm: A-reg}
For any sufficiently small positive $\e$
the metric $\l_{v, g}$ is $A$-regular, and
there are positive $\s, \delta$ such that $\sec(\l_{v,g})<0$. 
\end{thm}
\begin{proof}
Since the metric $\l_{v, g}$ is smooth, it is $A$-regular on
any compact subset, hence we can assume that $r\le p_n-\s$ 
so that $v=\e e^r$ and $g=\tau_\e+e^{r/2}$.

Denote $Y_0:=\ddr$. Arguing by induction on $k$, 
we shall show that for each integer $k\ge 0$
the components of $\nabla^k R$ in the frame 
$\{Y_0, Y_1, \dots, Y_{2n-1}\}$
are bounded functions of $r$ that 
have bounded derivatives with respect to $r$. 

Assume first that $k=0$.
The components of $(4,0)$-curvature tensor $R$ are sums
of sectional curvatures~\cite[Lemma 3.3.3]{Jos-3rd-ed}, 
while by (\ref{form: k(c,d)}), (\ref{form: k(c,d) non-generic}) 
the sectional curvature of any plane 
is a linear combination with constant coefficients
of terms in (\ref{form: k(y_i,y_1)})--(\ref{form: r(dr,y_1, y_2, y_3)}).
Furthermore, and this is really the key point, 
the terms in (\ref{form: k(y_i,y_1)})--(\ref{form: r(dr,y_1, y_2, y_3)})
as well as their derivatives by $r$
are obtained from the bounded (!) functions $\frac{1}{g}$ and $F$
by taking products, sums, and multiplying by real numbers; indeed we have:
\begin{eqnarray*}
\frac{v^\prime}{v}=1=\frac{v^{\prime\prime}}{v}
\ \ \ \ \text{and}\ \ \ \frac{g^\prime}{g}=F\ \ \ \text{and}\ \ \ 
F^\prime=\frac{F}{2}-F^2\\
\frac{g^{\prime\prime}}{g}=\frac{F}{2}\ \ \
\ \ \ \ \text{and}\ \ \ \ 
\frac{v}{g^2}=4\e F^2\ \ \ \ \text{and}\ \ \ \  
\left(\frac{1}{g^2}\right)^\prime=-2\frac{F}{g^2}.
\end{eqnarray*} 
It follows that the components of $R$ and their derivatives
are linear combinations of terms that are products of the 
functions $\frac{1}{g}$ and $F$, and hence are constant on any $r$-tube
and bounded in $r$.

For the induction step, we fix $k$ and let $S:=\nabla^k R$. 
The components of the tensor $\nabla S$ are 
\begin{eqnarray}
\label{form: derivative tensor}
\qquad Y_{i_0}(S(Y_{i_1},\dots, Y_{i_l}))-
\sum_{k=1}^l 
S(Y_{i_1},\dots, Y_{i_{k+1}},\nabla_{Y_{i_0}}Y_{i_k}, 
Y_{i_{k+1}}\dots, Y_{i_l}) 
\end{eqnarray} 
As discussed in~\cite[Appendix C]{Bel-rh-warp},
it follows from~\cite[Section 6]{BW} that 
\begin{eqnarray*}
& \nabla_\ddr \ddr=0=\nabla_\ddr Y_k\ \ \text{for $k\ge 1$,}\\ 
& \nabla_{Y_i} \ddr=\frac{g^\prime}{g}Y_i=FY_i\quad \text{for $i>1$,}\\
& \nabla_{Y_1} \ddr=\frac{v^\prime}{v}Y_1=Y_1
\end{eqnarray*}
By Section~\ref{sec: basis} one has
$[Y_i, Y_j]=c_{ij}\frac{v}{h^2}Y_1$ for $i, j>1$, and $[Y_i, Y_1]=0$
(at the point $z$ where we compute the curvature).
Plugging this into Koszul's formula~\cite[Appendix C]{Bel-rh-warp}
we compute $\nabla_{Y_k} Y_l$ as follows:
%
\begin{eqnarray*}
& \nabla_{Y_1} Y_1=-\frac{v^\prime}{v}\ddr=-\ddr,\quad\text{and}\quad
\nabla_{Y_i} Y_i =-\frac{g^\prime}{g}\ddr\ \ \text{if $i>1$,}\\
& \nabla_{Y_i} Y_j =c_{ij}\frac{v}{2g^2}Y_1\ \ \text{if $i,j>1$ are distinct}.
\end{eqnarray*}
So if $Y_{i_0}=\ddr$, then 
the induction hypothesis implies that the component
(\ref{form: derivative tensor}) is bounded, being a linear combination with
bounded coefficients of terms $S(Y_{i_1},\dots, Y_{i_l})$ or their derivatives.

If $Y_{i_0}\neq\ddr$, then by induction hypothesis
$S(Y_{i_1},\dots, Y_{i_l})$ is constant on $r$-tubes, so
$Y_{i_0}(S(Y_{i_1},\dots, Y_{i_l}))=0$, and again the remaining terms
are bounded by the induction hypothesis.

Thus the metric $\l_{v, g}$ is $A$-regular.
Next we show that $\sec(\l_{v,g})<0$ following the pattern of the proof of
Theorem~\ref{thm: main curv estimates}.
We only consider the generic case with the curvature given by
(\ref{form: k(c,d)}); the non-generic case is even easier 
because the mixed term in not present in (\ref{form: k(c,d) non-generic}). 

{\bf Step 1.} Suppose $r\in [o_\e, o_\e+\s]$ so that $v(r)=\e e^r$ and 
$\bg(r)=e^{r/2}$ and $g$ converges to $\bg$ in $C^1$-topology
as $\delta\to 0$, and furthermore,
$\frac{g^{\prime\prime}}{g}>\frac{1}{25}$ for small $\s$.
Hence $\frac{\bg^\prime}{\bg}=\frac{1}{2}$ and 
$\frac{v}{\bg^2}=\e$ implying 
$\frac{g^\prime}{g}>\frac{1}{3}$ and 
$\frac{v}{g^2}<2\e$.

Plugging into formulas 
(\ref{form: k(y_i,y_1)})--(\ref{form: r(dr,y_1, y_2, y_3)})
we get the following: 
\begin{eqnarray*} 
& 
K (Y_2,Y_1)=K(Y_3,Y_1)<
\frac{\e^2}{4}-\frac{1}{3}<-\frac{1}{4},\\
& 
K (Y_3,Y_2)<-\left(\frac{g^\prime}{g}\right)^2\le -\frac{1}{9},\\
& 
K(\ddr,Y_1)=-1, \ \ \ \ \
K(\ddr,Y_2)<-\frac{1}{25},\\
& 
|\langle R(\ddr , Y_1) Y_2, Y_3\rangle|\le 
|c_{23}|2\e(1-\frac{1}{3})< \e,
\end{eqnarray*}
and we finish as in Step $5$ of Theorem~\ref{thm: main curv estimates}.

{\bf Step 2.}
Suppose $r\in [p_\e-\s, o_\e]$ so that $v(r)=\e e^r$ and 
$\frac{\bg^\prime}{\bg}$ 
in increasing, and $g$ converges to $\bg$ in $C^1$-topology
as $\delta\to 0$, and also
$\frac{g^{\prime\prime}}{g}>\frac{1}{25}$ for small $\delta$.
As $\frac{\bg^\prime}{\bg}> F(p_n-\s)>\frac{1}{5}$ for small $\s$, 
we get $\frac{g^\prime}{g}>\frac{1}{5}$ and
$\frac{v\prime}{v}-\frac{g^\prime}{g}<\frac{4}{5}$
for small $\delta, \s$.

Since $(\ln(\bg))^{\prime\prime}>0$ on $[-\infty, p_\e]$ and $[p_\e, o_\e]$,
Lemma~\ref{lem: strict convex}, implies that $\ln(\bg)$ is strictly
convex on $[p_\e-\s, o_\e]$, hence the graph of $\ln(\bg)$
is above its tangent line at $o_\e$, i.e. $\ln(\bg(r))\ge r/2$,
so that $\bg(r)\ge e^{r/2}$. 
It follows that
$\frac{v}{\bg^2}\le \frac{\e e^r}{e^r}=\e$
so that $\frac{v}{g^2}<2\e$ for small $\delta$.
The above estimates combined with formulas 
(\ref{form: k(y_i,y_1)})--(\ref{form: r(dr,y_1, y_2, y_3)})
imply the following:
\begin{eqnarray*} 
& 
K (Y_2,Y_1)=K(Y_3,Y_1)<
\frac{\e^2}{4}-\frac{1}{5}<-\frac{1}{6},\\
& 
K (Y_3,Y_2)<-\left(\frac{g^\prime}{g}\right)^2< -\frac{1}{25},\\
& 
K(\ddr,Y_1)=-1, \ \ \ \ \
K(\ddr,Y_2)<-\frac{1}{25},\\
& 
|\langle R(\ddr , Y_1) Y_2, Y_3\rangle|\le 
|c_{23}|2\e\frac{4}{5}< \e,
\end{eqnarray*}
and we finish as in Step $5$ of Theorem~\ref{thm: main curv estimates}.

{\bf Step 3.} Suppose $r\le p_\e-\s$ so that $v(r)=\e e^r$ and
$g=\tau_\e+e^{r/2}$. We compute
that $\frac{g^\prime}{g}=F$, and $\frac{v}{g^2}=4\e F^2$, and
$\frac{g^{\prime\prime}}{g}=F/2$,
and deduce the following:
\begin{eqnarray*} 
& 
K (Y_2,Y_1)=K(Y_3,Y_1)=\e^2F^4-F<F(\e^2-1)<-\frac{F}{2},\\
& 
K (Y_3,Y_2)<-\left(\frac{g^\prime}{g}\right)=-F^2 \\
& 
K(\ddr,Y_1)=-1, \ \ \ \ \
K(\ddr,Y_2)=-\frac{F}{2},\\
& 
|\langle R(\ddr , Y_1) Y_2, Y_3\rangle|<
|c_{23}|4\e F^2\le 2\e F^2.
\end{eqnarray*}
Thus $K(C , D)+\frac{F}{2}\left((d_1c_2-d_2c_1)^2 + d_1^2c_3^2\right)+
\frac{F}{2} d_2^2c_0^2$ 
is bounded above by
\begin{eqnarray*}
&  
-d_1^2c_0^2 
-F^2d_2^2c_3^2 +6\e F^2 |d_1d_2c_0c_3|=\\
& \nonumber
-(|d_1c_0|-F|d_2c_3|)^2+
|d_1d_2c_0c_3|F(6\e F-2),
\end{eqnarray*}
in which every summand is nonpositive. Then the argument of 
Step $5$ of Theorem~\ref{thm: main curv estimates}
gives a function $M(\e,r)$ with
$K(C,D)\le M(\e,r)<0$ for all $C, D$, where 
$M(\e,r)\to 0$ as $r\to-\infty$ because e.g. 
$K(\ddr,Y_2)=-\frac{F}{2}\to 0$ as $r\to-\infty$.
\end{proof}

\section{Proof of Theorem~\ref{thm: main thm}}
\label{sec: proof of main thm}

Let $U$ be the intersection of $M\setminus S$ and a small
tubular neighborhood of $S$ 

To prove (iii), equip each component of $U$ with the 
negatively curved $A$-regular metric given by 
Theorem~\ref{thm: A-reg}. By construction the metric
extends the complex hyperbolic metric on $M\setminus U$,
and it is clearly negatively curved and $A$-regular, because so
is the complex hyperbolic metric. The metric has finite volume
by~\cite[Remark 3.3]{Bel-rh-warp}, and is complete because
outside some compact set it has the form $(I\times F, dr^2+g_r)$, 
where $I=[a,\infty)$, and $F$ is compact, namely
a circle bundle over a closed complex hyperbolic manifold.
Completeness can be checked outside a compact set, and 
$(I\times F, dr^2+g_r)$ is complete as
the total space of Riemannian submersion with compact fiber 
and complete base.

To prove (i), equip each component of $U$ with the 
negatively curved metric given by 
Theorem~\ref{thm: main curv estimates}.
By construction the metric
extends the complex hyperbolic metric on $M\setminus U$,
and its sectional curvature is clearly bounded above by a negative
constant. The metric has finite volume
by~\cite[Remark 3.3]{Bel-rh-warp}, and is complete
by the previous paragraph.

This metric will be used in proving (ii). 
Each end of $M\setminus S$ has a 
{\it cusp neighborhood $E$} which by definition means that 
$E$ admits a Riemannian submersion onto $(-\infty, 0]$, and
there exists a constant $K$ such that the ``holonomy'' diffeomorphism
$h_t$ from the fiber over $\{0\}$ to the fiber over $\{t\}$ 
is $K$-Lipschitz for each $t$.
(Indeed, each end of $M\setminus S$
corresponding to a cusp of $M$ has a neighborhood with warped product metric
$dr^2+f_r$ where $f_r$ is an almost flat metric
on the cusp cross-section, and the ``holonomy'' diffeomorphism
$h_t$ is $1$-Lipschitz by exponential convergence of geodesics.
Each end of $M\setminus S$ that approaches $S$ has a neighborhood
with metric
$dr^2+v^2d\theta^2+h^2 {\bf k}^{n-1}$.
In either case the $r$-coordinate projection is a Riemannian submersion
with compact fibers, and $h_t$ is $1$-Lipschitz as $v, h$ are 
increasing.) 
Then by~\cite[Theorem 4.2]{Bel-rh-warp} the group $\pi_1(M\setminus S)$
is hyperbolic relative to the fundamental groups of the ends of
$M\setminus S$.

\section{Proof of Corollary~\ref{thm: intro-gr-theoretic-cor}}
\label{sec: appl}
 
Most of the assertions are proved verbatim as 
in~\cite[Theorem 1.1]{Bel-rh-warp} 
with the following exceptions. 

(4) The claim follows from Theorem~\ref{thm: main thm}
and the Dehn Surgery Theorem for relatively hyperbolic 
groups~\cite{Osi-periph} (cf.~\cite{GroMan}) 
provided all peripheral subgroups are fully 
residually hyperbolic, 
i.e. if $H$ is peripheral, then for any finite subset $F\subset H$  
there is a homomorphism of $H$ onto a non-elementary hyperbolic group  
that is injective on $F$. 
Finitely virtually nilpotent subgroups are residually finite, 
hence fully residually hyperbolic.
Thus we can assume that $H$ maps onto a non-elementary
hyperbolic group with infinite cyclic kernel.  
Let $z$ generate the kernel. It suffices to
check that any finite subset $F$ is mapped injectively 
into $H/\langle z^n\rangle$ for {\it some} $n$, because the latter group
is finite-by-hyperbolic, and hence hyperbolic. 
If not, then for any $n$ there exist distinct
$s_n, s_n^\prime\in S$
that get identified in $Q_n$. Since $\langle z^n\rangle$ is the kernel,
$s_n, s_n^\prime$ we have 
$s_n^\prime=s_n z^{n k_n}$ for some integer $k_n\neq 0$.
But $S$ is finite, so only finitely many elements of $\langle z^n\rangle$
are obtained this way, i.e. $n k_n$ is a bounded sequence, 
which forces $k_n=0$ for large $n$ and gives a contradiction.

(5)
A group satisfies the {\it Strong Tits Alternative} 
if any subgroup either contains a nonabelian free group or 
is virtually abelian.
Tukia~\cite{Tuk} proved the following Tits Alternative for relatively 
hyperbolic groups: a subgroup that does not contain a non-abelian
free subgroup is either finite, or virtually-$\Z$, or lies in
a peripheral subgroup. Thus it suffices to check
the Strong Tits alternative for the peripheral subgroups.
If $M$ is compact, this is proved 
in~\cite[Theorem 1.1(6)]{Bel-rh-warp}, while if $M$ is noncompact,
then there exists a virtually nilpotent peripheral subgroup
that is not virtually abelian.

(6)
According to~\cite{Reb} a relatively hyperbolic group is biautomatic 
provided its peripheral subgroups are biautomatic. 
Virtually central extensions of hyperbolic groups
are biautomatic~\cite{NeuRee}.
Polycyclic subgroups of  a biautomatic group is virtually 
abelian~\cite{GerSho}, and in particular this applies to
finitely generated nilpotent groups,  
so we have to assume $M$ is compact.

(8)
$\pi_1(N)$ is not $CAT(0)$ even when $M$ is compact because centralizers
need not virtually split (and
for the same reason $\pi_1(N)$ does not act by semisimple
isometries on a $CAT(0)$ space)~\cite[Theorem 1.1 (iv), page 439]{BH}.
Indeed, consider any peripheral subgroup $H$ that is an extension
with hyperbolic quotient and infinite cyclic kernel generated by $z$.
Since $H$ is peripheral, the centralizer of $z$ in $\pi_1(N)$
lies in $H$, and hence coincides with $H$.
If the extension virtually splits, then the circle bundle 
would have the zero real first Chern class because 
if it were nonzero it would not vanish in a finite cover,
so this possibility is ruled out by the following.

\begin{lem}\label{lem: kahler form}
If $S$ is a compact totally geodesic complex $(n-1)$-submanifold
of a complete complex hyperbolic $n$-manifold $M$, then
the first Chern class of the normal bundle $\nu$ of $S$ in $M$
is nontrivial in real cohomology.
\end{lem}
\begin{proof}
To see that the circle bundle has nonzero real first Chern class,
look at the normal bundle $\nu$ of $S$ in $M$ and
note that by Whitney sum formula $c_1(\nu)$ is the difference between 
first Chern classes of $i^\#TM$ and $TS$ where $i\co S\to M$ is 
the inclusion. But $2\pi c_1$ is represented 
by the Ricci form~\cite[2.75]{Bes},
which equals to $-\frac{n+1}{2}$-multiple of the K\"ahler 
form~\cite[Remark after Theorem IX.7.5]{KN-II}
of the complex hyperbolic metric.
Since $S$, $M$ are complex hyperbolic, the K\"ahler form of $M$
restricts to the K\"ahler form of $S$. One then computes that 
$2\pi c_1(\nu)$ is represented by $-\frac{1}{2}$-multiple of
the K\"ahler form of $S$. Since $S$ is compact, the K\"ahler
class is nontrivial.
\end{proof}

(15) Farb~\cite{Far-relhyp} (see also~\cite{Osi-relhyp})  
proved that a relatively hyperbolic group has solvable word problem,
provided each peripheral subgroup has solvable word problem. 
Bumagin~\cite{Bum-conj} proved the same
for the conjugacy problem. 
Virtually central extensions of hyperbolic groups
are biautomatic~\cite{NeuRee}, in particular, they have solvable
word problem, and also solvable conjugacy problem~\cite{GerSho-autom}.
Finitely generated virtually nilpotent groups
are polycyclic-by-finite, and hence 
they are conjugacy separable~\cite{Rem, For},
which implies that they have 
solvable conjugacy problem~\cite{Mos}. 
Finally, polycyclic-by-finite groups have
solvable word problem because the property of having solvable 
word problem is preserved under extensions of finitely presented 
groups~\cite[Lemma 4.7]{Mil}.

\appendix

\section{Bending and smoothing convex functions}
\label{sec: appendix smooth convex}

This paper relies on delicate warped
product constructions, and I find it worthwhile to summarize
some elementary results on bending and smoothing convex functions. 

To avoid confusion we note that
in this paper a function $f$ is called {\it strictly convex}
if $f(tx+(1-t)y)<tf(x)+(1-t)f(y)$ for all $x\neq y$ and $t\in (0,1)$.
For example, if $f^{\prime\prime}>0$ everywhere, 
then $f$ is strictly convex, while the converse is not true:  
near $x=0$ the function $f(x)=x+x^4$ is increasing strictly 
convex, yet $f^{\prime\prime}(0)=0$.
Similarly, $f$ is called {\it strictly
concave} if $-f$ is strictly convex.

Lemma~\ref{lem: smoothing conv} 
modifies an argument of Ghomi~\cite{Gho} by keeping track
of the first and second derivative of the smoothing.

\begin{lem} \label{lem: smoothing conv}
Suppose that $f\co [a,b]\to\mathbb R$ is a 
positive continuous function such that 
for some $c\in (a,b)$ the restrictions of $f$ to $[a,c]$,
$[c,b]$ are $C^2$ with $f^{\prime\prime}>k$, and 
$f^\prime(c_-)\le f^\prime(c_+)$. 
Then \newline
\textup{(1)} for each small $\delta,\s >0$ there is a $C^2$ function 
$f_{\delta,\s}\co [a,b]\to\mathbb R$ that coincides with $f$
outside the $\s$-neighborhood of $c$, 
and satisfies $f_{\delta,\s}^{\prime\prime}>k$ on $[a,b]$.
\newline
\textup{(2)} if $f$ is $C^l$ with $0\le l\le \infty$ near $x\in [a,b]$, 
then $f_{\delta,\s}$ 
is $C^l$ near $x$, and $f_{\delta,\s}$ 
converges to $f$ in the uniform $C^l$-topology near $x$
as $\delta\to 0$. If $f$ is $C^\infty$ away from $c$, then
$f_{\delta, \s}$ is $C^\infty$ on $[a,b]$.
\newline 
\textup{(3)} If $f$ is increasing and $\delta$ is small enough, 
then $f_{\delta,\s}^\prime>0$. 
\end{lem}
\begin{proof} Let $\phi_\delta\co\mathbb R\to [0,1]$ 
be a smooth (bump) function with support within $(-\delta, \delta)$, 
and such that $\phi_\delta=1$ on $[-\delta/2, \delta/2]$. 
We use the following notations:
\[
\theta_{\delta}:=\frac{\phi_\delta}{\int_\mathbb R\phi_\delta}\ \ \
\text{and}\ \ \
g_{\theta_\delta}(x):=\int_\mathbb R g(x-y)\theta_\delta(y)dy=
\int_\mathbb R \theta_\delta(x-y)g(y)dy.
\]  
It is well known (see e.g. \cite[Theorem 2.3]{Hirsch-diff-top})
that $g_{\theta_\delta}$ is $C^\infty$, and 
if $g$ is $C^m$ with $0\le m\le\infty$,
then $g_{\theta_\delta}$ converges to $g$ in the uniform $C^m$-topology 
on $[a,b]$ as $\delta\to 0$; also if $m\ge 2$, then
convolution and differentiation commute,
so the $m$-th derivative of $g_{\theta_\delta}$
satisfies $(g_{\theta_\delta})^{(m)}=(g^{(m)})_{\theta_\delta}$.

Consider
\[
f_{\delta, \s}(x):=f_{\theta_\delta\theta_\delta}(x)\phi_{\s}(c-x)+
f(x)(1-\phi_{\s}(c-x)), 
\]
where $f_{\theta_\delta\theta_\delta}$ is the convolution
of $f$, $\theta_\delta$, and $\theta_\delta$.
Note that the order in which we convolve is irrelevant because
the operation is commutative and associative. 
Thus if $f$ is $C^l$ with $0\le l\le \infty$ near $x\in [a,b]$, 
then $f_{\delta,\s}$ is $C^l$ near $x$, and $f_{\delta,\s}$ 
converges to $f$ in the uniform $C^l$-topology near $x$ as $\delta\to 0$.
Note that $f_{\delta, \s}=f$ when $|x-c|>\s$, and 
$f_{\delta,\s}=f_{\theta_\delta\theta_\delta}$ when $|x-c|<\s/2$,
in particular, if $f$ is $C^\infty$ away from $c$, then
$f_{\delta, \s}$ is $C^\infty$ everywhere, which proves (2).

Convolution with any nonnegative function
preserves increasing functions, 
hence if $f$ is increasing, then so are $f_{\theta_\delta}$, 
$f_{\theta_\delta\theta_\delta}$.
Thus $f_{\delta,\s}^\prime$ either equals to 
$f_{\theta_\delta\theta_\delta}^\prime >0$, or 
converges to $f^\prime>0$ in $C^1$ topology as $\delta\to 0$,
hence (3) is proved.

Since $f_{\delta,\s}(x)$ converges to $f$ in $C^2$ topology
outside the $\s/4$-neighborhood of $c$, as $\delta\to 0$,  we know that
$f^{\prime\prime}_\delta(x)>k$ for small $\delta$ and $|x-c|>\s/4$. 
Since $f_{\delta,\s}=f_{\theta_\delta\theta_\delta}$
for $|x-c|<\s/2$, it remains to show that 
$f_{\theta_\delta\theta_\delta}^{\prime\prime}>k$.
To this end, let $q(x):=f(x)-k\frac{(x-c)^2}{2}$. 
Since $f^{\prime\prime}>k$ on $[a,c]$, $[c, b]$, respectively, 
the restrictions of $q$ to $[a,c]$ and $[c, b]$ satisfies
$q^{\prime\prime}>0$. 
Also $q^\prime(c_-)=f^\prime(c_-)\le f^\prime(c_+)=q^\prime(c_+)$, so
by Lemma~\ref{lem: strict convex}, $q$ is strictly convex on $[a,b]$, 
hence Lemma~\ref{lem: 2nd der is >0} implies that
$q_{\theta_\delta\theta_\delta}^{\prime\prime}>0$ on $[a,b]$, but
\[
q_{\theta_\delta\theta_\delta}^{\prime\prime}=
f_{\theta_\delta\theta_\delta}^{\prime\prime}-
\left(k\frac{(x-c)^2}{2}\right)_{\theta_\delta\theta_\delta}^{\prime\prime}=
f_{\theta_\delta\theta_\delta}^{\prime\prime}-k,
\]
proving (1).
\end{proof}

\begin{lem} \label{lem: strict convex}
If $a_1<c<a_2$, and
if $f_1, f_2$ are two strictly convex $C^1$ functions defined on
$[a_1,c]$, $[c,a_2]$ respectively such that $f_1(c)=f_2(c)$, and
$f_1^\prime(c_-)\le f_2^\prime(c_+)$, then the function 
$f\co [a_1, a_2]\to \mathbb R$ 
that equals to $f_1$ on $[a_1,c]$, and to $f_2$ on $[c,a_2]$ is
strictly convex.   
\end{lem}
\begin{proof}
Take $b_1\in [a_1,c)$,  $b_2\in (c, a_2]$, and show that
the line segment $[b_1, b_2]$ lies above the graph of $f$.
Let $\l_i$ be the line through $f(b_i), f(c)$, and 
$L_i$ be the tangent line to the graph of $f_i$ at $c$.
Since $f_1$ is strictly convex,
$\l_1> f_1> L_1$ on $[a_1,c)$ so
the slope of $\l_1$ is less than 
the slope of $L_1$, which equals to $f_1^\prime(c_-)$.
Similarly, strict convexity of $f_2$ implies that $\l_2> f_2> L_2$
on $(c, a_2]$, so the slope of $\l_2$ is
greater than the slope of $L_2$ 
which equals to $f_2^\prime(c_+)$.
Since $f_1^\prime(c_-)\le f_2^\prime(c_+)$, 
the slope of $\l_1$ is less than the slope
of $\l_2$, and hence the function $\l=\max\{\l_1, \l_2\}$ 
is strictly convex. Hence $[b_1, b_2]$
lies above the graph of $\l$ but strict convexity of $f_1, f_2$
implies that $f\le\l$, so $[b_1, b_2]$ lies above the graph of $f$.
\end{proof}

\begin{lem}\label{lem: 2nd der is >0}
If $f$ is strictly convex, 
then $f_{\theta_\delta\theta_\tau}^{\prime\prime}>0$. 
\end{lem}
\begin{proof}
Convolution with any nonnegative function
preserves strict convexity, so $f_{\theta_\delta}$, 
$f_{\theta_\delta\theta_\tau}$ are strictly convex.
Differentiating under the integral sign, we get that
$f_{\theta_\delta\theta_\tau}^{\prime\prime}$  
is the convolution of nonnegative smooth functions 
$f^{\prime\prime}_{\theta_\delta}$
and $\theta_\tau$. 
So if $f_{\theta_\delta\theta_\delta}^{\prime\prime}(x)=0$,
then $f^{\prime\prime}_{\theta_\delta}(x-y)$ must vanish
wherever $\theta_\delta(y)$ is nonzero, so 
$f^{\prime\prime}_{\theta_\delta}=0$ on a neighborhood of $x$.
It follows that $f_{\theta_\delta}$ is affine near $x$, 
which contradicts the strict convexity of $f_{\theta_\delta}$.
\end{proof}

The following modification of Lemma~\ref{lem: smoothing conv}
is useful.

\begin{prop} \label{prop: bend}
Given real numbers $k, a_1, c, a_2$ with $a_1<c<a_2$, 
let $f_1\co [a_1, c]\to\mathbb R$  and 
$f_2\co [c, a_2]\to\mathbb R$
be $C^2$ functions satisfying $f_i^{\prime\prime}\ge k$,
$f_1(c)=f_2(c)$
and $f_1^\prime (c)< f_2^\prime (c)$.
If $f\co [a_1, a_2]\to\mathbb R$ denotes the (continuous)
function satisfying $f=f_1$ on $[a_1, c]$ and  $f=f_2$ on $[c, a_2]$, 
then for any small $\delta>0$ there exists a $C^2$ function 
$f_\delta\co [a_1, a_2]\to\mathbb R$ such that \newline
\textup{(1)} $f_\delta^{\prime\prime}>k$\newline
\textup{(2)} $f_\delta=f$ and $f_\delta^\prime=f^{\prime}$ 
 at the points $a_1, a_2$, \newline 
\textup{(3)} if $f$ is increasing, then $f^\prime_\delta>0$ \newline 
\textup{(4)} If $f$ is $C^l$ on $[a_1,a_2]$ for some integer 
$l\in [0,\infty]$, then
$f_\delta$ is $C^l$ on $[a_1,a_2]$, and $f_\delta$
converges to $f$ in the $C^l$-topology on $[a_1,a_2]$ as $\delta\to 0$.
\end{prop}
\begin{proof}
Consider the functions
\[
F_{1,\delta}(r)=f_1(r)+\delta (r-a_1)^2,\ \ \ \
F_{2,\delta}(r)=f_2(r)+\delta (r-a_2)^2\frac{(c-a_1)^2}{(c-a_2)^2}
\]
defined on domains of $f_1, f_2$, respectively. 
For each $i$ the function $F_{i,\delta}(r)$ converges to $f_i$ 
in uniform $C^1$ topology on the domain of $f_i$, as $\delta\to 0$,
and furthermore,
$F^{\prime\prime}_{i,\delta}>k$ for small $\delta$,
and $F_{i,\delta}-f_i$ and $F_{i,\delta}^\prime-f_i^\prime$ vanish
at $a_i$. Also $F_{1,\delta}(c)=F_{2,\delta}(c)$,
and $F_{1,\delta}^\prime (c)< F_{2,\delta}^\prime (c)$ for 
small $\delta$.
Let $F_\delta$ be the (continuous) function satisfying 
$F_\delta=F_{i,\delta}$ on the domain of $f_i$.
Applying Lemma~\ref{lem: smoothing conv} to smooth $F_\delta$
near $c$, we get a function $f_\delta$ with required properties.
\end{proof}

\section{Curvature of warped product metrics}
\label{sec: components-of-curv-tensor}

In this appendix we review 
some formulas for the curvature tensor of a  
multiply-warped product metric $dr^2+g_r$ on $I\times F$ that were 
worked out in~\cite[Section 6]{BW}, and corrected in~\cite{Bel-rh-warp}. 

The computation in~\cite[Section 6]{BW}) works 
provided at each point $w$ of $F$ there is a basis of vector fields $\{X_i\}$
on a neighborhood $U_w\subset F$  that is $g_r$-orthogonal
for each $r$. We fix one such a basis for each $w$.
Let $h_i(r)=\sqrt{g_r(X_i,X_i)}$ so that $Y_i=X_i/h_i$
form a $g_r$-orthonormal basis on $U_w$ for any $r>0$.
Since $X_i\neq 0$ and $g_r$ is nondegenerate, $h_i>0$ 

To simplify some of the formulas below 
we denote $g(X,Y)$ by $\langle X , Y\rangle$, denote the vector field
$\frac{\d}{\d r}$ by $\ddr$,  and reserve the notation $\frac{\d}{\d r}T$ 
for the partial derivative of the function $T$ with respect to $r$.

A straightforward tedious computation
(done e.g. in~\cite[Section 6]{BW}) yields the following.
\begin{eqnarray}\label{form: curv of warped prod}
& \langle R_g (Y_i,Y_j) Y_j,Y_i\rangle=
\langle R_{g_r}(Y_i,Y_j) Y_j,Y_i\rangle -
\frac{h_i^\prime h_j^\prime}{h_ih_j},\\
& \langle R_g(Y_i,Y_j) Y_l,Y_m\rangle= 
\langle R_{g_r}(Y_i,Y_j) Y_l,Y_m\rangle
\ \ \ \mathrm{if}\ \{i,j\}\neq \{l,m\},\\
  & \langle R_g(Y_i,\ddr)\ddr ), Y_i \rangle=
-\frac{h_i^{\prime\prime}}{h_i},\ \ \ \ \ \langle
R_g(Y_i,\ddr)\ddr ), Y_j \rangle=0\ \ \ \mathrm{if}\ i\neq j.
\smallskip
\end{eqnarray}

The following mixed term is by far the most 
complicated and is usually the hardest to control:
by~\cite[Appendix C]{Bel-rh-warp} 
$2\langle R_g(\ddr ,Y_i) Y_j,Y_k\rangle$ equals to
\begin{eqnarray*}
\langle [Y_i,Y_j],Y_k\rangle
\left(\ln\frac{h_k}{h_j}\right)^\prime
+\langle [Y_k,Y_i],Y_j\rangle 
\left(\ln\frac{h_j}{h_k}\right)^\prime
+\langle [Y_k,Y_j],Y_i\rangle
\left(\ln\frac{h_i^2}{h_jh_k}\right)^\prime .
\end{eqnarray*}

\section{Acknowledgments}
It is a pleasure to thank Bill Goldman and Vitali Kapovitch 
for discussions relevant to this work, and the referee 
for valuable expository suggestions.
This work was partially supported by the NSF grant \# DMS-0503864.

\small
\bibliographystyle{amsalpha}
\bibliography{ch-warp-revised}
\end{document}